\documentclass[reqno,11pt]{amsart}

\usepackage[a4paper,left=25mm,right=25mm,top=30mm,bottom=30mm,marginpar=25mm]{geometry} 
\usepackage{amsmath}
\usepackage{amssymb}
\usepackage{amsthm}
\usepackage{amscd}
\usepackage[ansinew]{inputenc}
\usepackage{xcolor}
\usepackage[final]{graphicx}
\usepackage{tikz}
\usepackage{bbm}
\usepackage{ esint }
\usepackage{faktor}
\usepackage{subcaption}
\usepackage{soul}
\usepackage{tikz-cd}

\usetikzlibrary{patterns} 
\usetikzlibrary{positioning}
\usetikzlibrary{calc}
\usetikzlibrary{arrows,shapes,backgrounds}
\usetikzlibrary{intersections} 

\renewcommand{\epsilon}{\varepsilon}

\numberwithin{equation}{section}

\newtheoremstyle{thmlemcorr}{10pt}{10pt}{\itshape}{}{\bfseries}{.}{10pt}{{\thmname{#1}\thmnumber{ #2}\thmnote{ (#3)}}}
\newtheoremstyle{thmlemcorr*}{10pt}{10pt}{\itshape}{}{\bfseries}{.}\newline{{\thmname{#1}\thmnumber{ #2}\thmnote{ (#3)}}}
\newtheoremstyle{defi}{10pt}{10pt}{\itshape}{}{\bfseries}{.}{10pt}{{\thmname{#1}\thmnumber{ #2}\thmnote{ (#3)}}}
\newtheoremstyle{remexample}{10pt}{10pt}{}{}{\bfseries}{.}{10pt}{{\thmname{#1}\thmnumber{ #2}\thmnote{ (#3)}}}
\newtheoremstyle{ass}{10pt}{10pt}{}{}{\bfseries}{.}{10pt}{{\thmname{#1}\thmnumber{ A#2}\thmnote{ (#3)}}}

\theoremstyle{thmlemcorr}
\newtheorem{theorem}{Theorem}
\numberwithin{theorem}{section}
\newtheorem{lemma}[theorem]{Lemma}

\newtheorem{proposition}[theorem]{Proposition}

\theoremstyle{thmlemcorr*}
\newtheorem{theorem*}{Theorem}
\newtheorem{lemma*}[theorem]{Lemma}
\newtheorem{corollary*}[theorem]{Corollary}
\newtheorem{proposition*}[theorem]{Proposition}
\newtheorem{problem*}[theorem]{Problem}
\newtheorem{conjecture*}[theorem]{Conjecture}

\theoremstyle{defi}

\theoremstyle{remexample}
\newtheorem{remark}[theorem]{Remark}
\newtheorem{example}[theorem]{Example}

\theoremstyle{ass}

\newcommand{\Acal}{\mathcal{A}}

\newcommand{\Bcal}{\mathcal{B}}

\newcommand{\Mcal}{\mathcal{M}}

\newcommand{\Scal}{\mathcal{S}}
\newcommand{\Tcal}{\mathcal{T}}

\DeclareMathOperator{\id}{id}

\DeclareMathOperator{\dist}{dist}

\newcommand{\floor}[1]{\lfloor #1 \rfloor}

\newcommand{\norm}[1]{\|#1\|}

\newcommand{\dd}{\;\mathrm{d}}

\newcommand{\R}{\mathbb{R}}

\newcommand{\Z}{\mathbb{Z}}

\newcommand{\weakly}{\rightharpoonup}
\newcommand{\weaklystar}{\overset{*}\rightharpoonup}

\newcommand{\eps}{\epsilon}

\newcommand{\ffi}{\varphi}

\DeclareMathOperator{\SO}{SO}
\DeclareMathOperator*{\essup}{ess\,sup}

\DeclareMathOperator{\Id}{Id}

\newcommand{\eYrig}{Y^{\mathrm{rig}}_\eps}

\newcommand{\Yrig}{Y^{\mathrm{rig}}}

\def\Xint#1{\mathchoice 
{\XXint\displaystyle\textstyle{#1}}%
{\XXint\textstyle\scriptstyle{#1}}%
{\XXint\scriptstyle\scriptscriptstyle{#1}}%
{\XXint\scriptscriptstyle\scriptscriptstyle{#1}}%
\!\int} 
\def\XXint#1#2#3{{\setbox0=\hbox{$#1{#2#3}{\int}$} 
\vcenter{\hbox{$#2#3$}}\kern-.5\wd0}} 
\def\dashint{\,\Xint-}

\newcommand\restrict[1]{\raisebox{-.5ex}{$|$}_{#1}}

\newcommand{\hori}[2]{E_{#1}^{{#2},\to}}
\newcommand{\verti}[2]{E_{#1}^{{#2},\uparrow}}
\newcommand{\ui}[1]{^{(#1)}}

\title[Asymptotic analysis of fiber-reinforced materials]{Asymptotic analysis of deformation behavior in high-contrast fiber-reinforced materials: \\ Rigidity and anisotropy}

\author{Dominik Engl}
\address{Mathematisch Instituut, Universiteit Utrecht, Postbus 80010, 3508 TA Utrecht, The Netherlands}
\email{d.m.engl@uu.nl}

\author{Carolin Kreisbeck}
\address{Mathematisch-Geographische Fakult\"at, Katholische Universit\"at Eichst\"att-Ingolstadt, Osten\-stra{\ss}e 28, 85072 Eichst\"att}
\email{carolin.kreisbeck@ku.de}

\author{Antonella Ritorto}
\address{Mathematisch-Geographische Fakult\"at, Katholische Universit\"at Eichst\"att-Ingolstadt, Osten\-stra{\ss}e 28, 85072 Eichst\"att}
\email{antonella.ritorto@ku.de}

\begin{document}

 
\maketitle

\begin{abstract}  
	\vspace{-12pt} 
	 We identify the restricted class of attainable effective deformations in a model of reinforced composites with parallel, long, and fully rigid fibers embedded in an elastic body. In mathematical terms, we characterize the weak limits of sequences of Sobolev maps whose gradients on the fibers lie in the set of rotations. These limits are determined by an anisotropic constraint in the sense that they locally preserve length in the fiber direction. Our proof of the necessity emerges as a natural generalization and modification of the recently established asymptotic rigidity analysis for composites with layered reinforcements. However, the construction of approximating sequences is more  delicate here due to the higher flexibility and connectedness of the soft material component. We overcome these technical challenges by a careful approximation of the identity that is constant on the rigid components, combined with a lifting in fiber bundles for Sobolev functions. The results are illustrated with several examples of attainable effective deformations. If an additional second-order regularization is introduced into the material model, only rigid body motions can occur macroscopically. 

	\vspace{8pt}
	
	 
	 
	\noindent\textsc{Date:} \today.
\end{abstract}
	

\section{Introduction}\label{sec:introduction} 

Designing new composite materials with advanced mechanical features is an important agenda in the engineering sciences, with relevance for many branches of industry. The properties of a composite are tightly related to the characteristics and structure of its microscopic heterogeneities,  and it is well-known that the deformation behavior on a macroscopic scale may differ substantially from the way its components deform individually~\cite{Jon98,Mil02,VaM13}. 
A significant class of such materials with particular relevance for manufacturing lightweight structures are reinforced high-contrast composites; 
indeed, the combination of an elastically softer matrix medium with embedded stiff components of different shapes, such as~fibers, gives rise to a light, yet strong material. 

Motivated by these applications, an extensive body of mathematical literature on the homogenization of stiff fibered structures has emerged, providing various modeling approaches and techniques to pave the way for a reliable prediction of effective material response in reaction to external forces.
It is important to notice that scaling between the fiber thickness and the elastic properties plays a crucial role for the resulting homogenized model. 
We highlight here a few selected works. 
 In~\cite{BrE01, PiS97} and \cite{BrE07, Jar13}, the authors study
variational homogenization via $\Gamma$-convergence of Saint-Venant Kirchhoff energy functionals for vanishingly small fibers with suitable adhesion conditions and diverging Lam{\'e} coefficients in the elastically linear and nonlinear setting, respectively. 
As a consequence of the choice of scaling relations between the elastic constants, the fiber thickness and adhesive parameters in the papers~\cite{Jar13, PiS97}, the derived limit models describe second-gradient materials. 
Moreover, nonlocal effects have been observed to arise in models of homogenized fiber-reinforced structures, see e.g.,~\cite{BeB98} or \cite{BeG05, PaS16, Sil05} in the context of linear elasticity, as well as~\cite{Bel09}, where additional torsion effects are taken into account.

In this paper, we study a phenomenological model for composites reinforced by parallel long fibers, which are assumed to be fully rigid;
further modeling hypotheses are that the matrix material adheres to the fibers along all interfaces and that the relative volume percentage of the rigid components stays within fixed scale-invariant bounds. 
As we will see, the presence of rigid fibers gives rise to global restrictions of the material response, which sets this work apart from the aforementioned references dealing with stiff reinforcements.
Our goal here is to contribute to a qualitative understanding of the nonlinear model introduced in detail below by identifying, via a rigorous limit analysis, the class of anisotropic deformations that can be attained on a macroscopic scale.
\medskip

We begin the description of the model with the basic geometric set-up of the elastic body and the embedded fibers. 
Henceforth, let $\Omega=\omega\times (0,L)\subset \R^3$ be the reference configuration of a cylindrical body with height $L>0$ and cross section $\omega\subset \R^2$, where $\omega$ is a bounded Lipschitz domain. 
Whenever indicated, we assume additionally that $\omega$ satisfies the following hypothesis:
\begin{itemize}
\item[(H)] The domain $\omega$ is bi-Lipschitz homeomorphic to the open unit disk in $\R^2$, i.e., there exists a Lipschitz map $\phi : B(0,1)\to \omega$ whose inverse exists and is also Lipschitz.
\end{itemize}
Examples of sets with the property (H) include in particular rectangles (see e.g.~\cite{GHKR08}) 
or simply connected bounded domains with smooth boundary (see e.g.~\cite[Chapter~5.4]{Tay11}). 

To model the distribution of the fibers inside the body, consider a periodic lattice on $\R^2$ with unit cell 
$Y=[0,1)^2$ and a small length scale parameter $\eps>0$. We suppose that each scaled and translated cell $\eps(k+Y)$ with $k\in \Z^2$ contains the cross section of one fiber, described by a domain $\omega_\eps^k\subset \R^2$, which is assumed to satisfy two technical conditions. 
First, the fiber cross-sections are required to have (relative to their size) a fixed minimal distance to the boundary of their surrounding cell, precisely,
\begin{equation}\label{unif-compact-condition}
	\omega_\eps^k \subset \eps(k+ [\alpha,1-\alpha)^2)
\end{equation}
for a given $\alpha\in (0,\tfrac{1}{2})$.
As a second hypothesis, let each $\omega_\eps^k$ contain a square of side length $\eps\delta$ with fixed $\delta>0$ such that $\delta +2\alpha<1$, i.e., there is $a_\eps^k\in \eps(k+Y)$ such that
\begin{equation}\label{square-inside}
	S_{\eps}^k:= a_\eps^k + \eps(-\tfrac{\delta}{2},\tfrac{\delta}{2})^2   \subset  \omega_{\eps}^k;
\end{equation}
this guarantees that the measure of the fiber cross-sections scales like $\eps^{2}$, in particular, 
$|\omega_\eps^k|\geq \delta^2 \eps^2$. The two assumptions~\eqref{unif-compact-condition} and~\eqref{square-inside} are illustrated in Figure \ref{fig:setup}\,a).

\begin{figure}[h!]
		\centering
		\begin{subfigure}{.495\linewidth}
			\centering
			\begin{tikzpicture}
				\draw (0,0) rectangle (5,5);
				\draw (0.7,0.7) rectangle (4.3,4.3);
				\draw (4.2,5) node [anchor=south] {$\eps(k+Y)$};
				
				\draw [fill=gray!30!white] (1.5,1) [out=0,in=180] to (2.5,1.5) [out=0,in=225] to (3.8,1.3) [out=45,in=-90] to (3.8,2.6) [out=90,in=-90] to (4,3.7) [out=90,in=0] to (2.6,3.8) [out=180,in=0] to (1.5,4) [out=180,in=60] to (1,2) [out=240,in=180] to (1.5,1);
				\draw (1.5,1) node [anchor=south] {$\omega_\eps^k$};
				
				\draw [<->] (0,3.5) --++ (0.7,0);
				\draw [<->] (1.5,4.3) --++ (0,0.7);
				\draw (0.35,3.5) node [anchor=north] {$\eps\alpha$};
				\draw (1.5,4.65) node [anchor=west] {$\eps\alpha$};
				

				\draw (2.5,2.5) -- (3.5,2.5) -- (3.5,3.5) -- (2.5,3.5) -- cycle;
				\draw [<->] (2.4,2.5) -- (2.4,3.5);
				\draw (3,3) [fill=black] circle (1pt);
				\draw (3.1,3.25) node {$a_\eps^k$};
				\draw (2.4,3) node [anchor=east] {$\eps\delta$};
				\draw (3,2.5) node [anchor=north] {$S_\eps^k$};
			\end{tikzpicture}
			\put (-170,140) {a)}
		\end{subfigure}
		\begin{subfigure}{.495\linewidth}
			\centering
			\includegraphics[height=5.3cm]{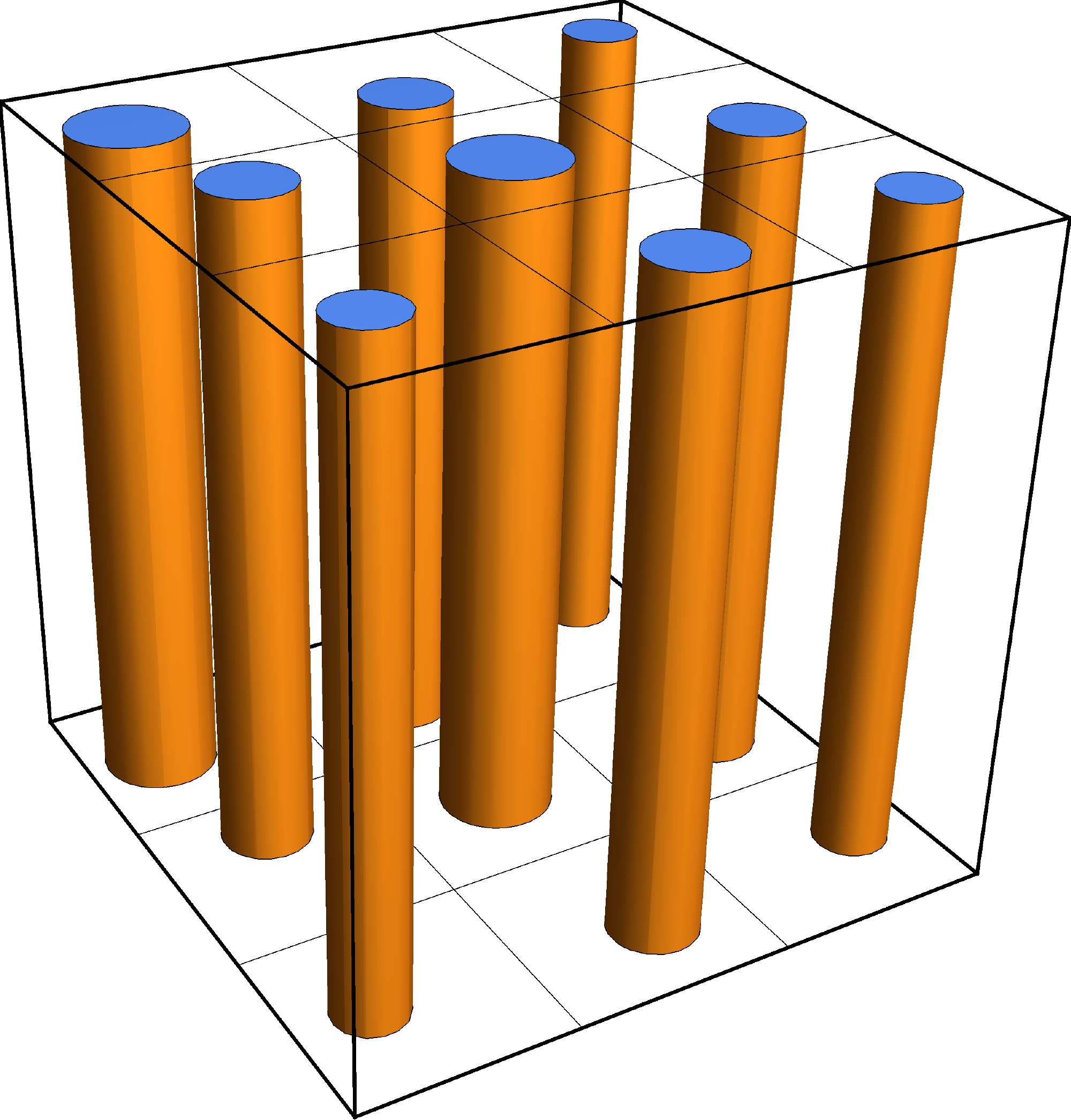}
			\put (-170,140) {b)}
			\put (-110,150) {$\Omega$}
			\put (-55,8) {$x_2$}
			\put (-133,25) {$x_1$}
			\put (-155,90) {$x_3$}
			\put (-25,80) {$\eYrig$}
		\end{subfigure}
		\caption{Illustration of: a) the scaled and translated unit cell $\eps(k+Y)$, which contains the fiber cross-section $\omega_\eps^k$ (shaded in grey) and the square $S_\eps^k$ in compliance with the minimal distance $\eps\alpha$ from the boundary; b) a collection of fibers embedded into the cuboid $\Omega$. }\label{fig:setup}
	\end{figure}
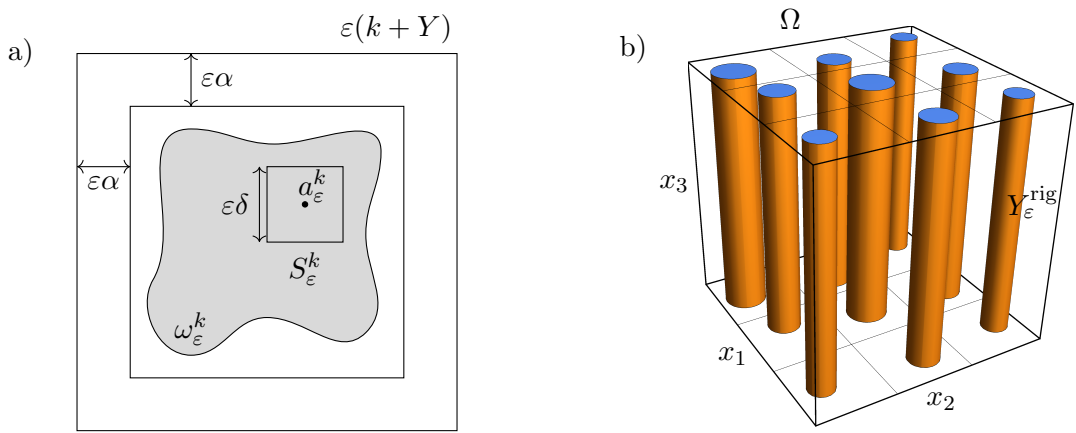

Considering the definition of the fiber cross-sections, we now introduce the set
\begin{align}\label{Yrig_eps}
	\eYrig = \bigcup_{k\in \Z^2} \omega^k_{\eps} \times \R;
\end{align}
then, $\eYrig\cap \Omega$ is the collection of all fibers in $\Omega$, and $\eYrig\cap \Omega$ corresponds to the matrix material, see Figure~\ref{fig:setup}\,b). 
Notice that in this set-up, the fiber cross-sections need not be periodically distributed.  
The analysis of a model where the fibers are distributed randomly without the confinement of the periodic lattice is an interesting problem, but  beyond the scope of this work. 

 With the geometric set-up in place, we now describe the possible deformations that a body with fibers $\Yrig_\eps\cap \Omega$ can undergo in response to external forces via Sobolev maps $u_\eps:\Omega\to \R^3$.
The rigidity of the fibers, which prevents any form of non-trivial elastic deformation, is reflected in the requirement that the deformation gradient $\nabla u_\eps$ restricted to each fiber is a local orientation-preserving isometry, or equivalently, by well-known rigidity results (cf.~e.g.~\cite{Res67}), a global rotation. Switching to the macroscopic point of view, one obtains the class of attainable effective deformations exactly as the weak limits of sequences $(u_\eps)_\eps$ when the scaling parameter $\eps$ tends to $0$. 
One may think of $(u_\eps)_\eps$ as a sequence of uniformly bounded energy for functionals of the form
\begin{align}\label{functional}
u\mapsto \int_{\Omega\setminus \eYrig} W_{\rm soft}(\nabla u) \dd{x} + \int_{\eYrig\cap \Omega}W_{\rm rig}(\nabla u)\dd{x},
\end{align}
where the elastic energy density $W_{\rm soft}:\R^{3\times 3}\to [0, \infty)$ satisfies standard properties including frame-indifference, suitable growth and coercivity assumptions, and vanishes on the identity, like, for instance, Saint Venant-Kirchhoff type densities; moreover, $W_{\rm rig}:\R^{3\times 3}\to [0, \infty]$ is given by $W_{\rm rig}(F) = 0$ for $F\in \SO(3)$ and $W_{\rm rig}=\infty$ otherwise, and can be interpreted as an energy density with infinitely large elastic constants. \medskip

The main result of this paper, stated in Theorem~\ref{theo:characterization}, is a complete characterization of the weak limits of such sequences $(u_\eps)_\eps$. 
It shows that the latter exhibit a restrictive anisotropic material response in the sense that the strain in the direction of the fibers merely depends on the cross-section variables, shows higher regularity, and most importantly, has unit length.  
Geometrically, this means that any vertical line in the reference configuration, which may be viewed as an infinitesimally thin fiber, can only be rotated and shifted. Several examples of macroscopically observable deformations are illustrated in Section~\ref{sec:examples}.

\begin{theorem}[Characterization of limit deformations]\label{theo:characterization}
Let $p>2$ and $\omega\subset \R^2$ be 
 a bounded Lipschitz domain that satisfies the hypothesis (H). With
	\begin{align*} 
		\Acal_\eps:= \{u\in W^{1,p}(\Omega;\R^3) : \nabla u \in \SO(3) \text{ a.e.~in } \eYrig\cap \Omega\}
	\end{align*}
	for $\eps>0$, the set of weak limits 
	\begin{align}\label{A_0}
		\Acal_0:=\{u \in W^{1,p}(\Omega;\R^3) : \text{there is $(u_\eps)_\eps$ with $u_\eps\in \Acal_\eps$ and $u_\eps \weakly u$ in $W^{1,p}(\Omega;\R^3)$}\}
	\end{align}
admits the three equivalent characterizations
	\begin{align*}
	\begin{split}
		\Acal_0 &= \{u\in W^{1,p}(\Omega;\R^3) : \partial_3 u \in W^{1,p}(\omega;\Scal^2)\}\\
		&=\{u\in W^{1,p}(\Omega;\R^3) : u(x) =x_3 \Sigma(x')+d(x')\text{ for a.e.~$x=(x', x_3)\in\Omega$}\\ & \hspace{3.7cm} \text{with $\Sigma\in W^{1,p}(\omega;\Scal^2), d\in W^{1,p}(\omega;\R^3)$}\}\\
		&=\{u\in W^{1,p}(\Omega;\R^3) : u(x) = R(x')x+b(x') \text{ for a.e.~$x=(x', x_3)\in\Omega$} \\ & \hspace{3.7cm}\text{with $R\in W^{1,p}(\omega;\SO(3)), b\in W^{1,p}(\omega;\R^3)$}\}.
		\end{split}
	\end{align*}
\end{theorem}
There are different facets to the general discussion of this result and its placement in the related literature we wish to mention: 
\begin{itemize}
\item[(i)] Theorem~\ref{theo:characterization} is a natural extension of the findings in the paper~\cite{ChK17} by Christowiak \& Kreisbeck from two to three dimensions. Considering that it is not possible to distinguish between layers and fibers in $2$d, 
also the reference~\cite{ChK20}, where the same authors study layered reinforcements, provides another natural three-dimensional extension of \cite{ChK17}. Whereas the codimension of the layer reinforcements in \cite{ChK20} is one, the codimension of the rigid components for our model with fibers is two, making the latter clearly more flexible. For more details, see Remark~\ref{rem:comparison}. 
\medskip

\item[(ii)] Generally speaking, the concept of asymptotic rigidity, as introduced in~\cite{ChK17, ChK20}, refers to global geometric constraints that emerge in the limit of functions that are (almost) local isometries on suitably arranged, and increasingly refined, disconnected parts of their domain, cf.~also~\cite{DFK19}. Analogues in the case that a domain consists of only one connected component are the well-known classical rigidity statements by Liouville for smooth and Reshetnyak~\cite{Res67} for Sobolev functions. In that spirit, Theorem~\ref{theo:characterization} represents asymptotic rigidity for fiber structures.
\medskip

\item[(iii)]  We point out that our main theorem provides a basis for future efforts regarding the homogenization via $\Gamma$-convergence of variational models for 
elastic materials reinforced with rigid long fibers. If one aims for a statement in analogy to \cite{ChK20}, where reinforcements in the form of rigid layers are studied, then Theorem~\ref{theo:characterization} implies that the $\Gamma$-limit for $\eps\to 0$ of energy functionals as in~\eqref{functional}, which are constrained   by the non-convex sets $\Acal_\eps$, are finite exactly on the limit set $\Acal_0$. In other words, the domain of the $\Gamma$-limit can be identified as a direct consequence. 
Notice, however, that the approximation we obtain from Theorem~\ref{theo:characterization} for elements in $\Acal_0$
will not be a recovery sequence in general. 
\medskip

\item[(iv)] Within a broader theoretical context of asymptotic analysis, one may interpret the previous theorem
as a $\Gamma$-convergence result for characteristic functions (in the sense of convex analysis) associated with $\Acal_\eps$, or equivalently, as the characterization of the Kuratowski limit of the sequence of sets $(\Acal_\eps)_\eps$, both with respect to weak convergence in $W^{1,p}(\Omega;\R^3)$; for more on these concepts, see~e.g.~\cite{Bra02, Dal93}.
\end{itemize}
 
The proof of Theorem~\ref{theo:characterization} falls naturally into two parts, the necessity and sufficiency, which we approach with different techniques. 
To see the necessity, it is possible to generalize and adapt the arguments in~\cite{ChK17, ChK20}, which again rely on a compactness result for sequences of piecewise constant rotations similar to~\cite[Theorem~4.1]{FJM02}. The required new ingredient is a suitable estimate for rotations on neighboring fibers, see Lemma~\ref{lem:approx_SO(3)}. 
The sufficiency, on the other hand, calls for an explicit construction of approximating sequences in $\Acal_\eps$.
Since the soft matrix components have an additional degree of freedom compared to the layer case and are connected (cf.\,(i) above), this construction is more involved. After evoking a lifting in fiber bundles for Sobolev functions, we consider a composition with a careful approximation of the identity that is constant on the rigid components (see Lemma~\ref{lem:approx-id}). 

The effect of higher-order regularizations in material models has been a subject of intense study, and especially, weaker penalizations that do not involve the full Hessian of the deformations (in the second-order case), have come into the focus more recently~\cite{BKS18, DFK19, DaF20}.  
In this spirit, we complement our model with a anisotropic partial regularization, precisely, a uniform bound on the second derivatives in the cross-section variables. The next theorem shows that this is already enough to deprive the macroscopic material response of any flexibility, that is, only rigid body motions can occur. 

\begin{theorem}[Rigid macroscopic behavior through partial regularization]\label{theo:rigidity_reg} 
Let $p>1$ 
and $u\in W^{1, p}(\Omega;\R^3)$. Suppose there exists a sequence $(u_\eps)_\eps\subset W^{1, p}(\Omega;\R^3)$ such that $u_\eps \in\Acal_\eps$ for all $\eps$ and 
	\begin{align}\label{second_derivatives}
	 \sup_{\eps} \max_{i, j\in \{1,2\}}\norm{\partial_i\partial_j u_\eps}_{L^p(\Omega;\R^3)}^p  <\infty,
	\end{align}
and $u_\eps \weakly u$ in $W^{1,p}(\Omega;\R^3)$ as $\eps\to 0$. 
Then, $u$ is a rigid body motion, i.e.,
	\begin{align*}
		u(x) = Rx + b, \quad x\in \Omega, 
	\end{align*}
with a rotation $R\in\SO(3)$ and a translation vector $b\in\R^3$. 	
\end{theorem}

This work is organized as follows. After introducing the relevant notations, we collect in Section~\ref{sec:preliminaries-tools} a few technical tools for working with manifold-valued, particularly $\SO(3)$-valued, Sobolev functions, including a lifting, extension, and density result. 
Section~\ref{sec:necessary_proof} is concerned with identifying necessary structural properties of weak limits of sequences $(u_\eps)_\eps$ with $u_\eps\in \Acal_\eps$. Moreover, we formulate in Proposition \ref{prop:dist-constraint} a generalization of Theorem \ref{theo:characterization}, where the exact differential inclusion $\nabla u_\eps\in \SO(3)$ a.e.~in $\eYrig\cap \Omega$ is replaced by a suitable approximate variant~(cf.~\eqref{approximate_inclusion}); from a modeling point of view, the latter makes the transition from rigid to elastically deformable, yet, very stiff fibers. 
In Section \ref{sec:sufficient_proof}, we detail the construction of suitable approximating sequences, showing the sufficiency part of Theorem~\ref{theo:characterization}.
In combination with the necessity from Section~\ref{sec:necessary_proof} and a lifting argument, the proof of Theorem~\ref{theo:characterization} is then completed. 
The intention of Section \ref{sec:examples} is to illustrate the obtained analytical results from a perspective of materials engineering. To this end, we present several examples of attainable macroscopic deformations, comment on conditions for incompressibility, and make a brief comparison with the setting of layered composites. 
Finally, the paper concludes with the proof of Theorem~\ref{theo:rigidity_reg} in Section \ref{sec:regularization}.

\section{Preliminaries and tools}\label{sec:preliminaries-tools}
\subsection{Notation} 
The standard unit vectors in $\R^n$ are $e_i$ for $i=1, \ldots, n$.
For $a,b \in \R^3$, we denote their cross product as $a \times b$ and their scalar product as $a\cdot b$. 
The two-dimensional unit sphere $\Scal^2$ consists of all unit vectors in $\R^3$.
On the matrix space $\R^{m\times n}$, we work with the standard Frobenius norm given by $|A| = \sqrt{\mathrm{Tr}(A^TA)}$ for $A\in \R^{m\times n}$, where $\mathrm{Tr}$ is the trace operator and $A^T\in \R^{n\times m}$ denotes the transpose of $A$. 
By $\SO(n)$, we denote the special orthogonal group of matrices in $\R^{n\times n}$, and $\Id_{\R^{n\times n}}$ stands for the identity matrix in $\R^{n\times n}$, while $\id_{\R^n}$ denotes the identity map on $\R^n$. For $t\in \R$, $\lceil t \rceil$ and $\lfloor t \rfloor$ are 
the smallest integer not less and the largest integer not greater than $t$, respectively. 
 
We write $B(x,r)$ for an open ball with center $x\in\R^n$ and radius $r>0$. 
If $A, B\subset \R^n$, then $A\Subset B$ means that $A$ is compactly contained in $B$. Moreover, we refer to an open, connected, and non-empty set $U\subset \R^n$ as a domain. The Lebesgue measure is denoted by $|\cdot|$ and $\#(\cdot)$ symbolizes the counting measure. 

We use the splitting $\R^3=\R^{2}\times \R$ and write $x=(x',x_3)$ with $x'=(x_1, x_2)\in \R^2$.
The projection onto the first two coordinates of the set $A\subset\R^3$ is denoted by $A' \subset \R^2$. 
Further, if $u: \R^3\supset U \to \R^m$ is  (weakly) differentiable, we split its (weak) gradient into $\nabla u = (\nabla 'u| \partial_3 u)$ with $\nabla' u = (\partial_1 u| \partial_2 u)$. For $d\in\R^3$, let $\partial_d u = (\nabla u)d$ be the directional derivative of $u$ in the direction $d$. 
In particular, if $d=e_k$, $k\in\{1,2,3\}$, we write $\partial_k u$ instead of $\partial_{e_k}u$.  
In case that $U\subset \R^2$ and $u: U\to \R^m$ is (weakly) differentiable, then the (weak) gradient of $u$ is denoted by $\nabla' u$. 

For $U\subset \R^n$ open and $1\leq p\leq \infty$, we use the standard notation for Lebesgue and Sobolev spaces,  $L^p(U;\R^m)$ and $W^{1,p}(U;\R^m)$. 
Replacing $\R^m$ by an embedded submanifold $\Mcal$ of $\R^m$, we set
\begin{align*}
	W^{1,p}(U;\Mcal):= \{u\in W^{1,p}(U;\R^m) : u \in \Mcal \text{ a.e.~in }U\},
\end{align*}
and analogously for the Lebesgue spaces. 
Without further mention, elements in $W^{1,p}(\omega;\R^3)$ are identified with functions in $W^{1,p}(\Omega;\R^3)$ via constant extension in $x_3$-direction. 

Finally, we use generic constants $C>0$ that may differ from one line to the other. 
Families indexed with $\eps>0$ refer to any sequence $(\eps_j)_{j}$ such that $\eps_j \to 0$ as $j\to\infty$. 

\subsection{Tools for manifold-valued Sobolev functions}
In this subsection, we collect a few auxiliary results. All the statements hold in more generality, but we present them here in a way that is tailored for applying them in later sections.

First, we present a lifting result for Sobolev functions with values in $\Scal^2$, which builds on the theory of fiber bundles, see e.g.~\cite{Ehr51, Ste99, Swi02} for selected references on the topic.
For the reader's convenience, we give the following definition:
Let $E,B,F$ be differentiable manifolds and $\pi:E\to B$ a differentiable map. The tuple $(E,B,F,\pi)$ is called a fiber bundle, if there exists an open cover $(U_\tau)_\tau$ of $B$ and diffeomorphisms $\phi_\tau : U_\tau \times F\to \pi^{-1}(U_\tau)$ such that $\pi\circ\phi_\tau$ 
projects canonically onto the first coordinate.
One usually refers to $E$ as the total space, $B$ is the base space, $F$ the fiber, $U_\tau$ a trivial neighborhood, and $\phi_\tau$ a trivialization.

Here, we are particularly interested in the projection 
\begin{align}\label{SO(3)_projection}
	\pi: \SO(3) \to \Scal^2,\quad R\mapsto Re_3,
\end{align}
which maps each rotation to its third column. 
In light of~\cite[Example 7.20~a), Theorem~7.15]{Lee03}, the map $\pi$ is a smooth submersion, and hence the tuple $(\SO(3),\Scal^2,\SO(2),\pi)$ is a fiber bundle, as a consequence of Ehresmann's lemma \cite{Ehr51}. 

With this fact in mind, the following lemma is a simple modification of the lifting result for Sobolev functions defined on the open unit disk in~\cite[Proposition~5]{BeC07}; the latter, in turn, builds on similar arguments as~\cite[Proposition 4.10]{Swi02}.

\begin{lemma}[Lifting of \boldmath{$\Scal^2$}-valued Sobolev functions]\label{lem:lifting}
	Let $\omega\subset \R^2$ be a bounded Lipschitz domain that satisfies the hypothesis (H) and let $\Sigma\in W^{1,p}(\omega;\Scal^2)$ for $p>2$. 
	Then, there exists $R\in W^{1,p}(\omega;\SO(3))$ such that $Re_3 = \Sigma$ a.e.~in $\omega$.
\end{lemma}
\begin{proof}
Let $\phi: B(0,1) \to \omega$ be the bi-Lipschitz map according to assumption (H).
Then,
	\begin{align*}
		\widetilde\Sigma := \Sigma \circ \phi \in W^{1,p}(B(0,1);\Scal^2),
	\end{align*}
	see e.g.~\cite[Corollary~1, page~303]{GoR90}. 
	Now, we apply \cite[Proposition 5]{BeC07} to find a function $\widetilde R \in W^{1,p}(B(0,1);\SO(3))$ such that
	$\widetilde R e_3 = \widetilde \Sigma$.
	The composition $R := \widetilde R \circ \phi^{-1} \in W^{1,p}(\omega;\SO(3))$ is then the sought lift; indeed,
	\begin{align*}
		Re_3 = \pi \circ R = \pi \circ \widetilde R \circ \phi^{-1} = \widetilde \Sigma \circ \phi^{-1} = \Sigma \circ \phi \circ\phi^{-1} = \Sigma,
	\end{align*}
	with $\pi$ as in~\eqref{SO(3)_projection}. 
\end{proof}

\begin{remark}[Non-uniqueness of lifts]\label{rem:lifting_formula}
This remark shows that the liftings obtained in Lemma \ref{lem:lifting} are generally not unique.
	To give an explicit example, let $\Sigma\in W^{1,p}(\omega;\Scal^2)$ for $p>2$ and assume that there is $\eta>0$ such that the intersection of the image $\Sigma(\omega)$ with the cylinder $B(0, \eta)\times \R\subset\R^3$ is empty.
	Under this assumption, $\Sigma_1^2 + \Sigma_2^2 = 1-\Sigma_3^2 \geq \eta^2$. 
	
	Then, the functions $R,S\in W^{1,p}(\omega;\SO(3))$ defined by
	\begin{align*}
		Re_3 = \Sigma ,\quad Re_2 =\frac{1}{\sqrt{\Sigma_{1}^2 + \Sigma_{2}^2}}(-\Sigma_{2},\Sigma_{1},0),\quad Re_1=Re_2\times Re_3,
	\end{align*}
	and
	\begin{align*}
		Se_3 = \Sigma ,\quad Se_2 =\frac{1}{\sqrt{\Sigma_{1}^2 + \Sigma_{2}^2}}(-\Sigma_{1}\Sigma_3,-\Sigma_{2}\Sigma_3,1-\Sigma_3^2),\quad Se_1=Se_2\times Se_3,
	\end{align*}
are both lifts of $\Sigma$ in the sense of Lemma~\ref{lem:lifting}.
\end{remark}

Next, we present a elementary extension result on Sobolev functions with values in $\SO(3)$, based on nearest point projections. The reader will notice that the result is also true more generally for compact embedded submanifolds of $\R^n$. 

\begin{lemma}[Sobolev-extension of \boldmath{$\mathbf{\SO(3)}$}-valued maps]\label{lem:Sobolev-extension}
	Let $2<p\leq \infty$ and $U\subset \R^2$ a bounded Lipschitz domain. 
	If $R\in W^{1,p}(U; \SO(3))$, then there exist an open set $V\subset \R^2$ with $U\Subset V$ and an extension $\bar{R} \in W^{1,p}(V; \SO(3))$ of $R$. 
\end{lemma}
\begin{proof}
According to standard Sobolev theory,
the function $R\in W^{1,p}(U;\SO(3))$ can be extended to an element in $W^{1,p}(\R^2;\R^{3\times 3})$ with compact support, which we denote again by $R$. 
From the tubular neighborhood theorem~\cite[Propositions 6.17 and 6.18]{Lee03}, we conclude the existence of an open bounded neighborhood $\Tcal\subset \R^{3\times 3}$ of $\SO(3)$ with $0\notin \Tcal$ and a smooth retraction map $r: \Tcal \to \SO(3)$, i.e., 
\begin{align*}
	r\restrict{\SO(3)}= {\rm id}_{\R^{3\times 3}}.
\end{align*}
By shrinking $\Tcal$ if necessary, we may assume that $r$ is smooth up to the boundary. 

Let us consider the preimage
	\begin{align*}
		 V=R^{-1}(\Tcal).
	\end{align*}
	We observe that $V$ is open, since $R$ is continuous by Sobolev embedding and that $V$ is bounded as a consequence of $0\notin\Tcal$ and the fact that $R$ vanishes outside of a bounded set. Besides, $U\subset R^{-1}(\SO(3))$ is compactly contained in $V$.
	
	Finally, we define $\bar{R}= r\circ R\restrict{V} : V \to  \SO(3)$, which, in view of 
	\begin{align*}
		\bar{R}\restrict{U}= r\restrict{SO(3)} \circ R\restrict{U} = R\restrict{U},
	\end{align*}
 	is indeed an extension of $R$.
	As $\bar R$ is the composition of a smooth function (up to the boundary) with a $W^{1,p}$-Sobolev map on the bounded set $V$, it follows that 
	$\bar{R}\in W^{1,p}(V;\SO(3))$. This proves the claim. 
\end{proof}

Finally, we state for the reader's convenience a special case of a well-known density result for manifold-valued Sobolev functions, see~e.g.~\cite[Theorem 2.1]{Haj09}.
\begin{lemma}[Density of smooth functions]\label{lem:approx_SO(3)}
	Let $U\subset \R^2$ be open and bounded. The set of smooth functions $C^\infty(U;\SO(3))$ is dense in $W^{1,p}(U;\SO(3))$ for all $p\geq 2$.
\end{lemma}

\section{Proof of necessity}\label{sec:necessary_proof}
Before we go into details of the asymptotic analysis of weakly convergent sequences as in the definition of $\Acal_0$ (see~\eqref{A_0}), 
let us point out a basic observation about the structure of elements $u_\eps \in \Acal_\eps$. 
As a consequence of the well-known rigidity result by Reshetnyak \cite{Res67}, 
it holds that $u_{\eps}$ is a rigid body motion on each connected component of $\eYrig\cap \Omega$, or in other words, on each individual fiber inside $\Omega$. 

The next lemma establishes the connection between neighboring fibers by estimating the difference between their associated rotations. It constitutes a generalization of \cite[Lemma 2.4]{ChK17} to three dimensions,  
a broader class of domains, and $p\geq 1$. 
For an illustration of the geometric set-up, see~Figure \ref{fig:parallelogram}.

\begin{lemma}\label{lem:neighboring_rotations} 
For given $m\in \R$ and $L_1, L_2, L_3>0$, let $E= E'\times (0,L_3)\subset \R^3$ with 
\begin{align}\label{parallelogram}
	E'=\{ (x_1,x_2)\in \R^2 \colon 0<x_1<L_1, \, mx_1 < x_2 < mx_1 +L_2 \}
\end{align}
and $d=(1+m^2)^{-1/2}(e_1+me_2)\in\R^3$.
Further, let $w_i \colon E \to \R^3$ for $i=1,2$ be affine functions given by $w_{i}(x) = A_{i}x + b_{i}$ for $x\in E$ with $A_{i} \in \R^{3\times 3}$ and $b_{i} \in \R^3$.

If $v\in W^{1,p}(E; \R^3)$ with $p\geq 1$ satisfies the partial boundary conditions
\begin{align}\label{boundaryvalues}
	v = w_{1} \text{ on } \partial E \cap \{ x_{1} = 0 \} \quad \text{ and }\quad v = w_{2} \text{ on } \partial E \cap \{ x_{1} = L_1 \}
\end{align}
in the sense of traces, then
\begin{align}\label{neighbor_affine}
	\int_E |\partial_d v|^p \dd x \geq \frac{C|E|L_3^p}{(1+|m|^p)L_1^p}|(A_2-A_1)e_3|^p
\end{align}
with a constant $C>0$ depending only on $p$.
\end{lemma}
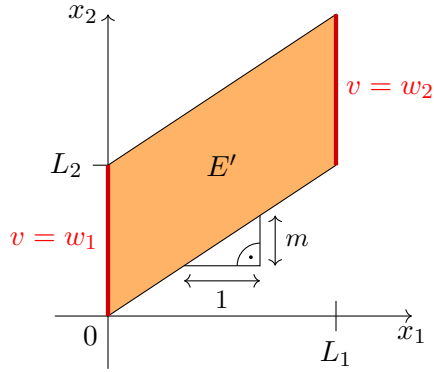
\begin{figure}[h!]
	\centering
	\begin{tikzpicture}
		\draw (-0.7,0)--(0,0);
		\draw (0,-0.7) --(0,0);
		\draw [->] (0,0) -- (4,0);
		\draw [->] (0,0) -- (0,4);
		\draw (4,0) node [anchor=north] {$x_1$};
		\draw (0,4) node [anchor=east] {$x_2$};
		\draw (3,-0.2) -- (3,0.2);
		\draw (3,-0.2) node [anchor=north] {$L_1$};
		\draw (-0.2,2) -- (0.2,2);
		\draw (0,0) node [anchor = north east] {$0$};
		\draw (-0.2,2) node [anchor=east] {$L_2$};
		
		\draw [fill=orange!60!white] (0,0) -- (0,2) -- (3,4) -- (3,2) -- cycle;
		\draw (1.5,2) node {$E'$};
		
		\draw [ultra thick, red!80!black] (0,0) -- (0,2);
		\draw [ultra thick, red!80!black] (3,4)--(3,2);
		\draw (0,1) node [red, anchor=east] {$v=w_1$};
		\draw (3,3) node [red, anchor=west] {$v=w_2$};
		
		\draw (1,0.66666) -- (2,0.66666) -- (2,1.33333);
		\draw (1.7,0.66666) arc (180:90:0.3);
		\draw [fill=black] (1.88,0.78) circle (0.5pt);
		\draw [<->] (1,0.46666) -- (2,0.46666);
		\draw (1.5, 0.46666) node [anchor=north] {\small $1$};
		\draw [<->] (2.2,0.66666) -- (2.2,1.33333);
		\draw (2.2, 1) node [anchor=west] {\small $m$};
	\end{tikzpicture}
	\caption{Illustration of the set $E'$ with slope $m$ as in \eqref{parallelogram} and the boundary values of $v$.}\label{fig:parallelogram}
\end{figure}
\begin{proof}
Since the inequality~\eqref{neighbor_affine} is continuous in $v$ with respect to the $W^{1,p}$-norm, it suffices by a density argument to prove the statement for smooth functions $v$ that attain the boundary values~\eqref{boundaryvalues} classically.

A change of variables gives 
\begin{align}\label{11}
\begin{split}
	\int_E |\partial_d v(x)|^p \dd x &=  \frac{1}{(1+m^2)^{p/2}}\int_Q |\partial_1 u(y)|^p\dd y\\
		&= \frac{1}{(1+m^2)^{p/2}}\int_{0}^{L_3} \int_{0}^{L_2}  \int_0^{L_1}|\partial_1 u(y_1,y_2,y_3)|^p \dd y_1 \dd y_2\dd y_3,
		\end{split}
\end{align}
where $u(y) = v(y_1,y_2 + my_1, y_3)$ for $y=(y_1,y_2, y_3)\in Q=(0,L_1)\times (0,L_2)\times (0,L_3)$. 
By Jensen's inequality, applied iteratively to the one-dimensional integrals over $(0, L_1)$ and $(0, L_2)$, we obtain in view of the assumption on the boundary values of $v$ in~\eqref{boundaryvalues} that
\begin{align}\label{x_1_minimized}
	\int_Q &|\partial_1u(y)|^p \dd x  \geq L_1^{1-p}\int_{0}^{L_3} \int_{0}^{L_2}|v(L_1,y_2+ mL_1,y_3)-v(0,y_2,y_3)|^p \dd y_2\dd y_3 \nonumber \\
	&\geq (L_1L_2)^{1-p}\int_{0}^{L_3} \Big| \int_0^{L_2} \sum_{k=2}^3 y_k (A_2-A_1) e_k  \dd y_2 + L_1 A_2 e_1 + mL_1 A_2 e_2 + b_2-b_1\Big|^p\dd y_3 \nonumber \\ 
	&\geq \frac{L_2}{L_1^{p-1}}\min_{b\in \R^3}\int_{0}^{L_3} |y_3(A_2-A_1)e_3 + b|^p\dd y_3. 
\end{align}

Next, we address the optimization problem in~\eqref{x_1_minimized}. The observation that any minimizer is parallel to $(A_2-A_1)e_3$, which follows from the elementary estimate
\begin{align*}
	|a + b|^2 \geq |a|^2 + |b|^2 - 2|a||b| 
	= \bigl|a - \tfrac{|b|}{|a|} a\bigr|^2\quad \text{for any $a, b\in \R^3$ with $a\neq 0$, }
\end{align*}
allows us to reduce~\eqref{x_1_minimized} to a one-dimensional minimization. Then, 
\begin{align}\label{second}
\begin{split}
	\min_{b\in \R^3}\int_{0}^{L_3} |y_3(A_2-A_1)e_3 + b|^p\dd y_3  = \min_{\lambda\in\R} \int_{0}^{L_3} |y_3+\lambda|^p |(A_2-A_1)e_3|^p\dd y_3.
	\end{split}
\end{align}

We join~\eqref{second} with \eqref{x_1_minimized} and~\eqref{11} to conclude that
\begin{align*}
	\int_E |\partial_d v(x)|^p \dd x &\geq \frac{L_3^{p+1}L_2}{2^p(p+1)(1+m^2)^{p/2}L_1^{p-1}}|(A_2-A_1)e_3|^p.
\end{align*}
This finishes the proof, considering that $|E|=L_1L_2L_3$ and $(1+m^2)^{p/2}\leq c (1+|m|^p)$ with a constant $c>0$ depending on $p$. 
\end{proof}

With these preparations, we can now prove the following proposition, which implies the necessity statement of Theorem \ref{theo:characterization}. The arguments combine ideas from the related papers~\cite{ChK17, FJM02} along with the new estimates from~Lemma~\ref{lem:neighboring_rotations}.
\begin{proposition}\label{prop:rigidity_general} 
Let $p> 1$ and suppose that $(u_\eps)_\eps \subset W^{1,p}(\Omega; \R^3)$ is a sequence with 
	\begin{align}\label{strictinclusion_prop}
		\nabla u_\eps \in \SO(3) \text{ a.e.~in } \eYrig\cap \Omega
	\end{align}
	for all $\eps$, such that 
 $u_\eps\weakly u$ in $W^{1,p}(\Omega; \R^3)$ for $u\in W^{1,p}(\Omega; \R^3)$.
	Then, 	
	\begin{align*}
		\partial_3 u \in W^{1,p}(\omega;\R^3) \text{ with } |\partial_3 u|=1 \text{ a.e.~in $\omega$,}
	\end{align*}
	or equivalently, there are $\Sigma\in W^{1,p}(\omega;\Scal^2)$ and $d\in W^{1,p}(\omega;\R^3)$ with 
	\begin{align*}
		u(x) = x_3\Sigma(x')  + d(x'), \quad x\in \Omega. 
	\end{align*} 
\end{proposition}
\begin{proof}
By an exhaustion argument, it suffices to prove for any cylindrical open set $U =U'\times (0, L)\subset \Omega$ with 
$U'\Subset \omega$ that $\partial_3 u\in W^{1,p}(U';\R^3)$ and $|\partial_3 u| = 1$ a.e.~in $U'$.  

We define for $\eps>0$,
\begin{align}\label{cuboids}
	P_{\eps}^{k}= \eps(k+ Y) \times (0, L) \quad \text{ for } k\in \mathbb{Z}^2,
\end{align}
and let the index set $I_\eps$ label the cuboids $P_\eps^k$ overlapping with $U$ 
i.e., $I_{\eps} =\{ k \in \Z^2 \colon P_\eps^k\cap U \neq \emptyset\}$.
Choosing $\eps>0$ sufficiently small, we may assume that 
\begin{align}\label{two_boxes}
	U\subset \bigcup_{k \in I_{\eps}} P_{\eps}^{k} \subset \Omega.
\end{align}

We split the rest of the proof in four steps.\medskip 

\textit{Step~1: Rigidity and approximation by piecewise affine functions.}
In this step, we tailor the strategy of~\cite[Proposition 2.1]{ChK17} to our setting, where the rigid components are thin in two directions instead of one. Precisely, for every $k \in I_{\eps}$, we apply the well-known rigidity result by Reshetnyak~\cite{Res67} in $\eYrig \cap P_{\eps}^{k}$ (recall $\eYrig$ from \eqref{Yrig_eps}), to find rotation matrices $R_{\eps}^{k} \in \SO(3)$ and translation vectors $b_{\eps}^{k} \in \R^3$ such that
\begin{align} \label{rigidity_on_fibers_0}
	u_{\eps}(x)=R_{\eps}^{k} x + b_{\eps}^{k} \quad \text{ for } x\in \eYrig \cap P_{\eps}^{k}.
\end{align}
For every $\eps>0$, we consider the auxiliary function $w_{\eps} = \sigma_{\eps} + b_{\eps} \in L^\infty(\Omega;\R^3)$ given by 
\begin{align*}
	\sigma_{\eps}(x) = \sum_{k \in I_{\eps}} (R^{k}_{\eps}x) \mathbbm{1}_{P_{\eps}^{k}}(x), 
		\quad \text{ and } \quad b_{\eps}(x) = \sum_{k\in I_{\eps}} b_{\eps}^{k} \mathbbm{1}_{P_{\eps}^{k}}(x), \quad x\in \Omega.
\end{align*}
We show that
\begin{equation}\label{u_e-w_e}
	\lim_{\eps \to 0} \| u_{\eps} - w_{\eps} \|_{L^{p}(U; \R^3)} = 0. 
\end{equation}
Indeed, with $u_{\eps}= w_{\eps}$ in $\eYrig \cap P_{\eps}^{k}$ for each index $k \in I_{\eps}$ and 
Poincar\'e's inequality applied in the cross-section variables, it follows that
\begin{align*}
	\int_{P_{\eps}^{k}} |u_{\eps}-w_{\eps}|^p \dd x &
		= \int_0^L\int_{(\eYrig \cap P_{\eps}^{k})'} |u_{\eps}-w_\eps|^p \dd x' \dd x_3\\
	&\le C\eps^p \int_0^L \int_{(\eYrig \cap P_{\eps}^{k})'} |\nabla' u_{\eps}-\nabla' w_\eps|^p \dd x'\dd x_3 \\
	&\le C\eps^p \int_{\eYrig \cap P_{\eps}^{k}} |\nabla u_{\eps}-R^{k}_{\eps}|^p \dd x \le C\eps^p \int_{P_{\eps}^{k}} |\nabla u_{\eps}-\nabla w_\eps|^p \dd x\\
	&\le C\eps^{p} \bigl( \|\nabla u_{\eps}\|_{L^p(P_{\eps}^{k}; \R^{3\times 3})}^p + |P_{\eps}^{k}|\bigr);
\end{align*}
here, we have used in particular that $u_\eps - w_\eps = 0$ on $\partial \omega_\eps^k= \partial(\eYrig \cap P_{\eps}^{k})'$ in the sense of traces and that the Poincar\'e constant scales linearly with the diameter of the domain.
Summing over all $k\in I_{\eps}$ then yields
\begin{align*}
	\| u_{\eps} - w_{\eps} \|_{L^{p}(U; \R^3)}^p \le C\eps^{p} \bigl( \| u_{\eps}\|_{W^{1,p}(\Omega; \R^3)}^p + |\Omega|\bigr)
\end{align*}
in light of \eqref{two_boxes}. Since $(u_{\eps})_{\eps}$ is bounded in $W^{1,p}(\Omega;\R^3)$ as a weakly convergent sequence, we conclude \eqref{u_e-w_e}. 

As a consequence, $(w_{\eps})_{\eps}$ is bounded in $L^p(U;\R^3)$ and, thus, as $(\sigma_{\eps})_{\eps}$ is uniformly bounded in $L^\infty(U;\R^3)$, the sequence $(b_{\eps})_{\eps}$ is bounded in $L^{p}(U;\R^3)$ as well. 
Therefore, (after passing to non-relabeled subsequences) there exist limit functions $\sigma \in L^{\infty}(U; \R^3)$ and $\hat{b}\in L^p(U; \R^3)$ such that 
$\sigma_{\eps} \weaklystar \sigma$ in $L^{\infty}(U; \R^3)$ and $b_{\eps} \rightharpoonup \hat{b}$ in $L^p(U; \R^3)$. 
Since this implies $w_{\eps}\rightharpoonup \sigma + \hat{b}$ in  $L^p(U; \R^3)$, we finally deduce the identity 
\begin{align}\label{sum_of_two}
	u = \sigma +\hat{b},
\end{align}
in light of the weak convergence of $(u_\eps)_\eps$ and \eqref{u_e-w_e}; note that $\hat b$ is independent of $x_3$.\medskip

\textit{Step~2: Compactness of an auxiliary function.}
For $\eps>0$, define $\Sigma_{\eps}\in L^{\infty}(\omega;\Scal^2)$ by  
\begin{align}\label{Sigma_e}
	\Sigma_{\eps} (x') = \sum_{k \in I_{\eps}} R^{k}_{\eps}e_3 \mathbbm{1}_{\eps(k + Y)}(x'),\quad x'\in \omega,
\end{align}
with $R_\eps^k$ as in \eqref{rigidity_on_fibers_0}. We will show with the help of the Fr\'echet-Kolmogorov compactness theorem and the estimates of Lemma \ref{lem:neighboring_rotations} that $(\Sigma_\eps)_\eps$ has a strongly convergent subsequence.
The proof strategy is inspired by and follows the lines of~\cite[Theorem 4.1]{FJM02}. The main difference in our approach is the estimates for the rotations on two neighboring cells, as derived in Lemma~\ref{lem:neighboring_rotations}.

Let us introduce the following sets: For $\eps>0$ and $k\in \Z^2$, we take $\hori{\eps}{k},\verti{\eps}{k}\subset \R^2$ to be the open parallelograms determined by the two parallel lines
\begin{align}\label{E1}
	a_\eps^k+ \eps\left(\{\tfrac{\delta}{2}\} \times (-\tfrac{\delta}{2},\tfrac{\delta}{2})\right) \quad \text{and} \quad
		a_\eps^{k+e_1} +\eps\left(\{-\tfrac{\delta}{2}\}\times (-\tfrac{\delta}{2},\tfrac{\delta}{2})\right),
\end{align}
and
\begin{align}\label{E2}
	a_\eps^k + \eps\left((-\tfrac{\delta}{2},\tfrac{\delta}{2})\times\{\tfrac{\delta}{2}\}\right) \quad \text{and} \quad
		a_\eps^{k+e_2} + \eps\left((-\tfrac{\delta}{2},\tfrac{\delta}{2})\times \{-\tfrac{\delta}{2}\}\right),
\end{align}
respectively; for simplicity, we restrict ourselves to the special case when the centers $a_\eps^k$ of the squares $S_\eps^k$ (cf.~\eqref{square-inside}) are periodically arranged and given by $a_\eps^k = \eps(k+a)$
with a fixed $a\in [\alpha+\tfrac{\delta}{2}, 1-\alpha-\tfrac{\delta}{2})^2$. 
This specific choice of $a_\eps^k$ means that the $\hori{\eps}{k}$ and $\verti{\eps}{k}$ are in fact rectangles, as illustrated~in Figure~\ref{fig:parallelograms33}. In Remark~\ref{rem:parallelogram} below, we explain how the arguments can be modified to cover the general case.

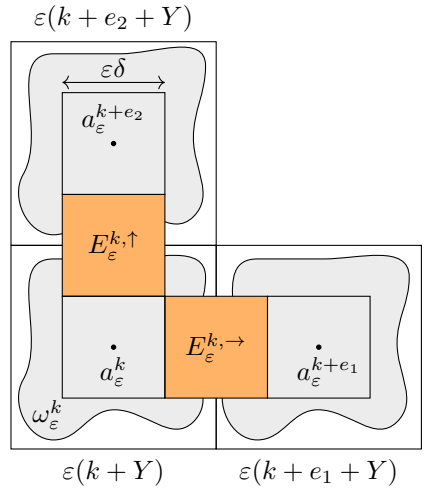
\begin{figure}[h!]
	\centering
	\begin{tikzpicture}
		\begin{scope}[scale=0.8, shift={(-0.8,-0.85)}]
 			\draw [fill=gray!15!white] (1.5,1) [out=0,in=180] to (2.5,1.5) [out=0,in=225] to (3.8,1.3) [out=45,in=-90] to (3.8,2.6) [out=90,in=-90] to (4,3.7) [out=90,in=0] to (2.6,4) [out=180,in=0] to (1.5,4) [out=180,in=60] to (1,2) [out=240,in=180] to (1.5,1);	
 		\end{scope}
 		\begin{scope}[scale=0.8, shift={(-0.8,2.55)}]
 			\draw [fill=gray!15!white] (1.5,1) [out=0,in=180] to (2.5,1.5) [out=0,in=225] to (3.8,1.3) [out=45,in=-90] to (3.8,2.6) [out=90,in=-90] to (4,3.8) [out=90,in=0] to (2.6,4) [out=180,in=0] to (1.5,4) [out=180,in=60] to (1,2) [out=240,in=180] to (1.5,1);	
 		\end{scope}
 		\begin{scope}[scale=0.8, shift={(2.55,-0.85)}]
 			\draw [fill=gray!15!white] (1.5,1) [out=0,in=180] to (2.5,1.5) [out=0,in=225] to (3.8,1.3) [out=45,in=-90] to (3.8,2.6) [out=90,in=-90] to (4,3.7) [out=90,in=0] to (2.6,4) [out=180,in=0] to (1.5,4) [out=180,in=60] to (1,2) [out=240,in=180] to (1.5,1);	
 		\end{scope}
		\begin{scope}[scale=2.7]
			\draw (0,0) rectangle (1,1);
			\draw (1,0) rectangle (2,1);
			\draw (0,1) rectangle (1,2);
			\draw (0.5,0) node [anchor=north] {\small $\eps(k+Y)$};
			\draw (1.5,0) node [anchor=north] {\small $\eps(k+e_1+Y)$};
			\draw (0.5,2) node [anchor=south] {\small $\eps(k+e_2+Y)$};
			
			\draw [fill=orange!60!white] (0.75,0.25) rectangle (1.25,0.75);
			\draw [fill=orange!60!white] (0.25,0.75) rectangle (0.75,1.25);
			\draw (0.5,1) node {\small $\verti{\eps}{k}$};
			\draw (1,0.5) node {\small $\hori{\eps}{k}$};
			
			\draw (0.25,0.25) rectangle (0.75,0.75);
			\draw (1.25,0.25) rectangle (1.75,0.75);
			\draw (0.25,1.25) rectangle (0.75,1.75);
			\draw [fill=black](0.5,0.5) circle (0.3pt);
			\draw [fill=black](1.5,0.5) circle (0.3pt);
			\draw [fill=black](0.5,1.5) circle (0.3pt);
			\draw (0.5,0.5) node [anchor=north] {\small $a_\eps^k$};
			\draw (1.55,0.5) node [anchor=north] {\small $a_\eps^{k+e_1}$};
			\draw (0.5,1.5) node [anchor=south] {\small $a_\eps^{k+e_2}$};
			\draw (0.18,0.17) node {\small $\omega_\eps^k$};
			
			\draw [<->] (0.25,1.8) --++ (0.5,0);
			\draw (0.5,1.87) node {\small $\eps\delta$};
 		\end{scope}
	\end{tikzpicture}
	\caption{Illustration of a scaled and translated unit cell and two of its neighbors; the rectangles $\hori{\eps}{k}$ and $\verti{\eps}{k}$ connect the horizontally and vertically neighboring squares $S_\eps^k=a_\eps^k + \eps(-\textstyle\frac{\delta}{2},\frac{\delta}{2})^2$, respectively. }
	\label{fig:parallelograms33}
\end{figure}	

Moreover, let $N^k_{\eps}$ be the union of $\eps(k + Y)$ and its eight neighboring cells, i.e.,
\begin{align*}
	N_{\eps}^k = \bigcup_{e \in J_k} \eps (e + Y) \quad \text{with}\ J_k= k+\{0, \pm e_1, \pm e_2, (\pm 1, \pm 1) \}\subset \Z^2;
\end{align*}
observe that $N_\eps^k\subset \omega$ for $\eps$ sufficiently small.

Let $x'\in \eps(k+Y)$  and $y' \in N_{\eps}^k$ with $k\in I_\eps$. Then,
\begin{align*}
	|\Sigma_{\eps}(x') -\Sigma_{\eps}(y')|^p \leq C \Bigl(\sum_{e\in J_k\cap (J_k-e_1)} |(R^{e+e_1}_{\eps} - R^{e}_{\eps}) e_3|^p +  \sum_{e\in J_k\cap(J_k-e_2)} |(R^{e+e_2}_{\eps} - R^{e}_{\eps}) e_3|^p \Bigr).
\end{align*} 
To each term on the right-hand side, we now apply Lemma \ref{lem:neighboring_rotations}, or the analogon thereof for switched roles of the variables $x_1$ and $x_2$, with $m=0$ and the domains $E=\hori{\eps}{k}\times(0,L)$ and  $E=\verti{\eps}{k}\times(0,L)$ (cf.~\eqref{E1} and~\eqref{E2}), respectively, to find that
\begin{align}\label{est998}
	|\Sigma_{\eps}&(x') -\Sigma_{\eps}(y')|^p \nonumber\\
		&\leq C\eps^{p-2}\Bigl(\sum_{e\in J_k\cap (J_k-e_1)} \norm{\partial_1 u_\eps}_{L^p(\hori{\eps}{e}\times(0,L);\R^3)}^p + \sum_{e\in J_k\cap(J_k-e_2)}\norm{\partial_2u_\eps}_{L^p(\verti{\eps}{e}\times(0,L);\R^3)}^p \Bigr) \nonumber \\
	&\leq C \eps^{p-2}\norm{\nabla' u_\eps}_{L^p(N_\eps^k\times(0,L);\R^{3\times 2})}^p,
\end{align}
where $C>0$ depends on $L$, $p$ and $\delta$.

Now, let $\xi \in \R^2$ be such that $|\xi| < \frac{1}{2} \dist(U', \partial \omega)$ and set $m_{\eps}=\lceil \frac{|\xi|_{\infty}}{\eps} \rceil$ with $|\xi|_{\infty}:=\max\{|\xi_1|, |\xi_2|\}$. Choosing $m_{\eps} +1$ points $0=\xi\ui{0}, \xi\ui{1}, \ldots, \xi\ui{m_{\eps}}= \xi$ such that $|\xi\ui{j+1}-\xi\ui{j}|_{\infty} \leq \eps$ for every $j=0, \ldots, m_{\eps}-1$ generates a discrete path from the origin to $\xi$ with maximal step width $\eps$, and it follows via a telescoping sum argument and the discrete H\"older's inequality that
\begin{align*}
	|\Sigma_{\eps}(x') - \Sigma_{\eps}(x'+\xi)|^p \le m_{\eps}^{p-1} \sum_{j=0}^{m_{\eps}-1} |\Sigma_{\eps}(x' +\xi\ui{j} ) - \Sigma_{\eps}(x' + \xi\ui{j+1})|^p. 
\end{align*}
After integration over $\eps(k+Y)$ and along with~\eqref{est998}, one obtains
\begin{align*} 
	\int_{\eps(k + Y)} |\Sigma_{\eps}(x' +\xi ) - \Sigma_{\eps}(x')|^p \dd x' \leq C m_\eps^{p-1}\eps^{p}\sum_{j=0}^{m_\eps-1}\norm{\nabla'u_\eps}_{L^p((N_\eps^k+\xi\ui{j})\times(0,L);\R^{3\times 2})}^p.
\end{align*}
By summing over all $k\in I_{\eps}$, we infer in view of $m_{\eps}\leq 2 \frac{|\xi|}{\eps}+1$ and the boundedness of the sequence $(\nabla'u_\eps)_\eps \subset L^p(\Omega;\R^{3\times 2})$ that
\begin{align}\label{estimate_frechet}
	\begin{split}
		&\int_{U'} |\Sigma_{\eps}(x' +\xi ) - \Sigma_{\eps}(x')|^p \dd x' \\  
		& \qquad\qquad \le C m_\eps^{p-1}\eps^{p}\sum_{j=0}^{m_\eps-1}\sum_{k\in I_\eps}\norm{\nabla'u_\eps}_{L^p((N_\eps^k+\xi\ui{j})\times(0,L);\R^{3\times 2})}^p \\
		&\qquad\qquad \leq C m_\eps^{p}\eps^{p}\norm{\nabla'u_\eps}_{L^p(\Omega;\R^{3\times 2})}^p \leq C \left(|\xi|^p + \eps^{p}\right),
	\end{split}
\end{align}
for $\eps>0$ sufficiently small, with a constant $C>0$ depending on $L$, $p$ and $\delta$.

Hence, the Fr\'echet-Kolmogorov theorem (see e.g.~\cite[Theorem 4.16]{Alt16}) implies the existence of $\Sigma\in L^p(U';\R^3)$ and a subsequence (not relabeled) of $(\Sigma_\eps)_\eps$ such that 
\begin{align}\label{Sigma_convergence}
	\Sigma_\eps \to \Sigma \text{ in $L^p(U';\R^3)$ and also pointwise a.e.~in $U'$.}
\end{align} 
Since $|\Sigma_\eps|=1$ in $U'$ (cf.~\eqref{Sigma_e}), the limit function $\Sigma$ satisfies
\begin{align*}
	|\Sigma|=1\text{ a.e.~in $U'$.}
\end{align*}\medskip

\textit{Step~3: Regularity of $\Sigma$.}
If we divide~\eqref{estimate_frechet} by $|\xi|^p$ and take the limit $\eps\to 0$, it follows under consideration of~\eqref{Sigma_convergence} that 
\begin{align}\label{difference_quotient}
	\left\|\frac{\Sigma(\cdot +\xi) - \Sigma}{|\xi|}\right\|_{L^p(U';\R^3)}^p  \leq C,
\end{align}
which implies $\Sigma\in W^{1,p}(U';\R^3)$.

Notice that the above-mentioned exhaustion argument exploits that the constant in \eqref{difference_quotient} is independent of $U'$.
\medskip

\textit{Step~4: Properties of the limit deformation $u$.} 
Recall that $u=\sigma + \hat{b}$ with $\partial_3 \hat b=0$ according to~\eqref{sum_of_two}. Then, \begin{align}\label{x_3_derivative}
	\partial_3 u = \partial_3 (\sigma + \hat b) = \partial_3 \sigma = \Sigma,
\end{align}
where the last identity follows from the observation that $\partial_3 \sigma_{\eps} =\Sigma_\eps\to \Sigma$ by
~\eqref{Sigma_convergence}. 

As an immediate consequence of~\eqref{x_3_derivative},
\begin{align*} 
	\partial_3 u \in W^{1,p}(U'; \R^3) \text{ with $|\partial_3u| =1\text{ a.e.~in } U'$,}
\end{align*}
or equivalently, $u(x) = x_3\Sigma(x') + d(x')$ with some $d\in W^{1,p}(U';\R^3)$. This concludes the proof.
\end{proof}

\begin{remark}[General distribution of fiber cross-sections]\label{rem:parallelogram}
	In Step 2 of the previous proof, it is assumed for simplicity that the squares $S_\eps^k$ inside the fiber cross-sections $\omega_\eps^k$ are periodically distributed. This remark addresses the necessary adaptations in order to cover the non-periodic case, where the sets $\hori{\eps}{k}$ and $\verti{\eps}{k}$ in~\eqref{E1} and~\eqref{E2}  may be general parallelograms. 
	Indeed, in this case, we can derive the estimate
	\begin{align}\label{frechet_slopes}
		\int_{U'} |\Sigma_\eps(x'+ \xi)-\Sigma_\eps(x')|^p \dd x' \leq C\sup_{k\in I_\eps}(1+\max\{|m_\eps^{k,\to}|,|m_\eps^{k,\uparrow}|\}^p)(|\xi|^p+\eps^p),
	\end{align}
	where $C>0$ depends on $L$, $\delta$ and $p$; here, the quantities $m_\eps^{k,\to},m_\eps^{k,\uparrow}\in\R$ are the slopes corresponding to $\hori{\eps}{k}$ and $\verti{\eps}{k}$, respectively, cf.~ Lemma \ref{lem:neighboring_rotations}. 
	Under the assumptions \eqref{unif-compact-condition} and \eqref{square-inside}, it holds that
	\begin{align}\label{max_slope}
		\sup_{k\in I_\eps} \max\{|m_\eps^{k,\to}|,|m_\eps^{k,\uparrow}|\} \leq \frac{1-2\alpha-\delta}{2\alpha},
	\end{align}
	which follows from a simple geometric argument, see Figure \ref{fig:max_slope}. 
	Thus, combining \eqref{max_slope} with \eqref{frechet_slopes} yields \eqref{estimate_frechet} with a constant depending on $L$, $p$ and $\delta$, as well as on $\alpha$.

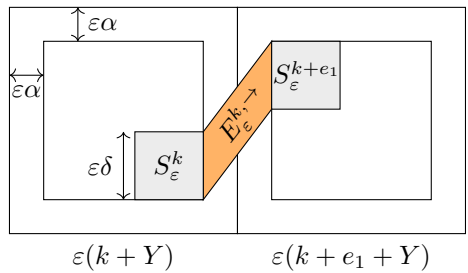
\begin{figure}[h!]
	\centering
	\begin{tikzpicture}
		\begin{scope}[scale=3]
			\draw (0,0) rectangle (1,1);
			\draw (1,0) rectangle (2,1);
			\draw (0.5,0) node [anchor=north] {\small $\eps(k+Y)$};
			\draw (1.5,0) node [anchor=north] {\small $\eps(k+e_1+Y)$};
			
			\draw [fill=orange!60!white] (0.85,0.15)--(0.85,0.45) -- (1.15,0.85) -- (1.15,0.55) -- cycle;
			\draw (1.025,0.525) node[rotate=52] {\small $\hori{\eps}{k}$};
			
			\draw (0.15,0.15) rectangle (0.85,0.85);
			\draw (1.15,0.15) rectangle (1.85,0.85);
			
			\draw [fill=gray!15!white] (0.55,0.15) rectangle (0.85,0.45);
			\draw [<->](0.5,0.15) --++ (0,0.3);
			\draw (0.5,0.3) node [anchor=east] {\small $\eps\delta$};
			\draw [fill=gray!15!white](1.15,0.55) rectangle (1.45,0.85);
			\draw (0.7,0.3) node {\small $S_\eps^k$};
			\draw (1.305,0.7) node {\small $S_\eps^{k+e_1}$};

			\draw [<->] (0,0.7) -- (0.15,0.7);
			\draw (0.075,0.7) node [anchor=north] {\small $\eps\alpha$};
			\draw [<->] (0.3,0.85) -- (0.3,1);			
			\draw (0.3,0.925) node [anchor=west] {\small $\eps\alpha$};
 		\end{scope}
	\end{tikzpicture}
	\caption{Illustration of $\hori{\eps}{k}$ in the non-periodic case. The corresponding slope is determined by the side length $\eps\delta$ of $S_\eps^k$ and the parameter $\alpha$ as in \eqref{unif-compact-condition}.}\label{fig:max_slope}
\end{figure}
\end{remark}

The following shows the exact differential inclusion~\eqref{strictinclusion_prop} in Proposition~\ref{prop:rigidity_general}
can be weakened to an approximate one, namely to~\eqref{approximate_inclusion} as below, without changing the result. 
Nevertheless, we have decided to provide the reader with both proofs, as they use different techniques and Section~\ref{sec:regularization} builds on the arguments of Proposition \ref{prop:rigidity_general}.

\begin{proposition}\label{prop:dist-constraint}
Let $p>1$ and suppose that $(u_\eps)_\eps\subset W^{1,p}(\Omega;\R^3)$ is a sequence satisfying  
	\begin{align}\label{approximate_inclusion}
		\int_{\eYrig\cap\Omega} \dist^p(\nabla u_{\eps},\SO(3)) \dd x \leq C\eps^\beta
	\end{align}
	for all $\eps$, with a constant $C>0$ and $\beta>2p$. 
	If $u_\eps\weakly u$ in $W^{1,p}(\Omega;\R^3)$ for $u\in W^{1,p}(\Omega;\R^3)$, then there exist $\Sigma\in W^{1,p}(\omega;\Scal^2)$ and $d\in W^{1,p}(\omega;\R^3)$  such that
	\begin{align}\label{sigma_d_repr8}
		u(x) = x_3\Sigma(x') + d(x'),\quad x\in \Omega.
	\end{align}
\end{proposition}
\begin{proof} 
We use the same basic set-up as in the proof of Proposition~\ref{prop:rigidity_general}. Recall in particular that $U=U'\times (0,L)\subset \Omega$ with $U'\Subset \omega$, $P_\eps^k$ from~\eqref{cuboids} for $\eps>0$ and $k\in\Z^2$, and the index set $I_\eps$; furthermore, $\eps$ is supposed to be  small enough such that \eqref{two_boxes} holds.

The rest of the proof is organized in two steps. \medskip
	
	\textit{Step 1: Fiber-wise approximation by rigid body motions.} We apply the quantitative geometric rigidity estimate by Friesecke, James \& M\"ulller (\cite[Theorem 3.1]{FJM02}) to each stiff sub-fiber $S_\eps^k\times (0, L)$ for $k\in I_\eps$ with $S_\eps^k$ as in~\eqref{square-inside}. 
	This yields rotations $R_\eps^k\in\SO(3)$ such that
	\begin{align}\label{rigidity_on_fibers}
		\norm{\nabla u_\eps - R_\eps^k}_{L^p(S_\eps^k\times (0,L);\R^{3\times 3})} \leq C \eps^{-2} \norm{\dist(\nabla u_{\eps},\SO(3))}_{L^p(S_\eps^k\times (0, L))},
	\end{align}
with a constant $C>0$ independent of $k$ and $\eps$. Notice that the scaling factor $\eps^{-2}$ in~\eqref{rigidity_on_fibers} is correlated with the size of the squares $S_\eps^k$. Indeed, this follows from a related scaling analysis for objects that are thin in one dimension as in~\cite[Theorem~6]{FJM06} and~\cite[Theorem~3.2.6]{Chr18}, if we evoke the argument twice in two different directions.  

Let $w_\eps := \sigma_\eps + b_\eps \in L^p(\Omega;\R^3)$ with
\begin{align*}
	\sigma_\eps(x) = \sum_{k\in I_\eps} R_\eps^k x \mathbbm{1}_{P_\eps^k}(x)\quad \text{and}\quad b_\eps(x) = \sum_{k\in I_\eps} b_\eps^k \mathbbm{1}_{P_\eps^k }(x),\quad x\in\Omega,
\end{align*}
where $b_\eps^k = \dashint_{S_\eps^k\times(0, L)} u_\eps(x) - R_\eps^kx \dd x$. This choice of translations implies 
\begin{align*}
\int_{S_\eps^k\times (0,L)} u_\eps - w_\eps \dd x=0,
\end{align*} and hence, enables an application of Poincar\'e's inequality with mean-value condition on the sub-fibers. In doing so, we obtain
	\begin{align}\label{estimate_on_fibers}
		\norm{u_\eps - w_\eps}_{L^p(S_\eps^k\times (0, L);\R^{3})}\leq  C \norm{\nabla u_\eps - R_\eps^k}_{L^p(S_\eps^k\times (0, L);\R^{3\times 3})}
	\end{align}
with a constant $C>0$ that does not depend on $\eps$ and $k$.
	
Next, we utilize the previous estimates on (parts of) the stiff fibers in order to control the difference $u_\eps-w_\eps$ also on the soft compoments. 
To this end, we cover each $P_\eps^k$ with $k\in I_\eps$ by finitely many shifted versions of $S_\eps^k\times (0, L)$; in view of~\eqref{unif-compact-condition} and \eqref{square-inside}, there are translation vectors $d_{\eps,n}^k\in\R^2$ with $n=1,\ldots,N:=\lceil\delta^{-2}\rceil$ satisfying 
\begin{align}\label{covering_soft}
	P_\eps^k \subset \bigcup_{n=1}^N (S_\eps^k+ d_{\eps,n}^k) \times (0, L)
\end{align}
up to a set of zero measure. Then, 
\begin{align*}
	\int_{ (S_\eps^k+ d_{\eps,n}^k) \times (0, L)} |u_\eps - w_\eps|^p \dd x &\leq C\int_{S_\eps^k \times (0, L)} |u_\eps - w_\eps|^p \dd x\\
		&\quad\quad \qquad+C\int_{S_\eps^k \times (0, L)}|(u_\eps-w_\eps)(x) - (u_\eps-w_\eps)(x + d_{\eps,n}^k)|^p \dd x\\
		&\leq C\big(\norm{u_\eps - w_\eps}^p_{L^p(S_\eps^k \times (0, L);\R^3)} + |d_{\eps,n}^k|^p\norm{\nabla' u_\eps - \nabla'w_\eps}^p_{L^p(P_\eps^k;\R^{3\times 2})}\big),
\end{align*}
and consequently, in view of~\eqref{covering_soft}, 
	\begin{align*}
		\int_{P_\eps^k} |u_\eps - w_\eps|^p \dd x \leq C\bigl(\norm{u_\eps - w_\eps}_{L^p(S_\eps^k \times (0, L);\R^3)}^p + \eps^p\norm{\nabla u_{\eps}}_{L^p(P_\eps^k;\R^{3\times 3})}^p + \eps^p|P_\eps^k|\bigr). 
	\end{align*}
	
	Summing this inequality over all $k\in I_\eps$, we conclude together with~\eqref{estimate_on_fibers} and \eqref{rigidity_on_fibers} that
	\begin{align*}
		\int_U |u_\eps - w_\eps|^p \dd x \leq C \bigl(\eps^{-2p}\norm{\dist(\nabla u_\eps,\SO(3))}^p_{L^p(\eYrig\cap \Omega)} + \eps^p\norm{u_\eps}_{W^{1,p}(\Omega;\R^3)}^p + \eps^p |\Omega|\bigr),
	\end{align*}
and finally in light of assumption \eqref{approximate_inclusion}, 
	\begin{align}\label{w_e-u_e_new}
		\norm{u_\eps - w_\eps}_{L^p(U;\R^3)} \leq C(\eps^{\frac{\beta}{p}-2} + \eps),
	\end{align}
	where the constant $C>0$ is independent of $\eps$. 
	
	Since  $\beta>2p$, this shows in particular that the sequences $(w_\eps)_{\eps}$ and $(u_\eps)_\eps$ have an identical weak $L^p$-limit, namely $u$.\medskip 
	
	\textit{Step 2: Compactness result.} We consider the functions
	\begin{align*}
		\Sigma_\eps(x') = \sum_{k\in I_\eps} R_\eps^ke_3 \mathbbm{1}_{\eps(k+Y)}(x'),\quad x'\in \omega,
	\end{align*}
	with $R_\eps^k$ as in Step 1, and observe that $\Sigma_\eps = \partial_3\sigma_\eps= \partial_3 w_\eps$.
	
	Let $\xi\in\R^2$ with $|\xi|< \frac{1}{2}\dist(U',\partial \omega)$.
	In analogy to the optimization argument in the proof of Lemma~\ref{lem:neighboring_rotations},
	\begin{align}\label{123}
		\begin{split} &\norm{w_\eps(\cdot + \xi) - w_\eps}_{L^p(U;\R^3)}^p = \sum_{k\in I_\eps} \int_{P_\eps^k\cap U} |(R_\eps^{k+\floor{\frac{\xi}{\eps}}} - R_\eps ^k)x +  R_\eps^{k+\floor{\frac{\xi}{\eps}}}\xi + b_\eps^{k+\floor{\frac{\xi}{\eps}}} - b_\eps^k|^p \dd x\\
		&\qquad \qquad\geq C\sum_{k\in I_\eps} |R_\eps^{k+\floor{\frac{\xi}{\eps}}}e_3 - R_\eps ^ke_3|^p | P_\eps^k \cap U| \geq C\norm{\Sigma_\eps(\cdot + \xi) - \Sigma_\eps}_{L^p(U';\R^3)}^p,\end{split}
	\end{align}
	where $\floor{\eta} = (\floor{\eta_1},\floor{\eta_2})$ for $\eta\in \R^2$. 
	The left-hand side in~\eqref{123} can be estimated from above by
	 \begin{align*}
	 	\norm{w_\eps(\cdot + \xi)-w_\eps}_{L^p(U;\R^3)} &\leq \norm{w_\eps(\cdot + \xi) - u_\eps(\cdot + \xi)}_{L^p(U;\R^3)} + \norm{u_\eps(\cdot + \xi) - u_\eps}_{L^p(U;\R^3)}\\
	 		&\qquad\qquad\quad + \norm{w_\eps - u_\eps}_{L^p(U;\R^3)}  \\
	 	&\leq 2\norm{w_\eps - u_\eps}_{L^p(U;\R^3)} + |\xi|\norm{\nabla'u_\eps}_{L^p(\Omega;\R^{3\times 2})}.
	 \end{align*}

	 Hence, it follows along with \eqref{w_e-u_e_new} that
	 \begin{align*}
	 	\norm{\Sigma_\eps(\cdot + \xi) - \Sigma_\eps}_{L^p(U;\R^3)} \leq C(|\xi| + \eps^{\frac{\beta}{p}-2} + \eps),
	 \end{align*}
	with $C>0$ independent of $\eps$. 
	The Fr\'echet-Kolmogorov theorem \cite[Theorem 4.16]{Alt16} then implies once again that
	there exists $\Sigma\in L^p(U';\R^3)$ and a non-relabeled subsequence of $(\Sigma_\eps)_\eps$ with
	\begin{align*}
		\Sigma_\eps \to \Sigma \text{ in $L^p(U';\R^3)$ and pointwise a.e. in $U'$,}
	\end{align*}
and consequently, $\Sigma\in\Scal^2$ a.e.~in $U'$.\medskip
	
	 \textit{Step~3: Conclusion.} The statement follows immediately, if we repeat Steps~3 and~4 in the proof of Proposition \ref{prop:rigidity_general}.
\end{proof}

We conclude this section with a brief comment on the validity of Proposition~\ref{prop:dist-constraint} for scaling exponents $\beta\leq 2p$.

\begin{remark}[Optimal scaling exponent]\label{rem:optimality}	Note that the statement of Proposition \ref{prop:dist-constraint} is not true for $\beta\leq p$ in general. If we consider for simplicity a suitable cuboid $\Omega\subset \R^3$, then this is a direct consequence of the related theory of layered composites in~\cite[Section~2 and Corollary~3.8]{ChK20}.

To see this, we first introduce the collection of rigid layers
	\begin{align}\label{layers}
		X_\eps^{\rm rig} = \bigcup_{i\in\Z} \eps(i + [\alpha,1-\alpha))\times \R^2
	\end{align}
	for $\eps>0$ with $\alpha\in (0, \frac{1}{2})$ as in~\eqref{unif-compact-condition} and observe that $\eYrig \subset X_\eps^{\rm rig}$.
	Following~\cite[Example 2.3 and Lemma~2.1]{ChK20}), let $u_\eps:\Omega\to \R^3$ be a Lipschitz function that induces uniform bending of all stiff layers in $\Omega$; naturally, $u_\eps$ also deforms all fibers in $\Omega$ in the same way.
	The elastic energy contribution of $u_\eps$ on the stiff components can be estimated by
	\begin{align*}
		\int_{\eYrig\cap\Omega}\dist^p(\nabla u_\eps,\SO(3)) \dd x \leq \int_{X_\eps^{\rm rig}\cap\Omega}\dist^p(\nabla u_\eps,\SO(3)) \dd x \leq C \eps^p,
	\end{align*}
	where the constant $C>0$ is independent of $\eps$. 
	Since the weak $W^{1,p}$-limit of $(u_\eps)_\eps$ cannot be expressed in the form \eqref{sigma_d_repr8}, it is confirmed that Proposition \ref{prop:dist-constraint} fails for $\beta\leq p$. 
	
	Whether Proposition~\ref{prop:dist-constraint} is valid in the scaling regimes $\beta\in (p,2p]$, though, remains an interesting open problem.
\end{remark}

\section{Proof of sufficiency in Theorem~\ref{theo:characterization}}\label{sec:sufficient_proof}
The main ingredient for the construction of approximating sequences in the proof of Theorem~\ref{theo:characterization} is a suitable approximation of the identity in two dimensions that is constant on the cross section of the fibers. The following lemma can be viewed as a generalization of the one-dimensional result in~\cite[Lemma~4.3]{ChK17}. 

\begin{lemma}[Approximation of the identity]\label{lem:approx-id} 
Let $U\subset \R^2$ be a bounded Lipschitz domain and $\alpha \in (0,\frac{1}{2})$. Further, let $\omega_\eps^k\subset \R^2$ with $\eps>0$ and $k\in \Z^2$ be open domains satisfying~\eqref{unif-compact-condition}, i.e., 
\begin{align*}
	\omega_\eps^k\subset \eps (k + [\alpha, 1-\alpha)^2).
\end{align*}
   
Then, there exists a sequence $(\ffi_{\eps})_{\eps} \subset W^{1, \infty}(U; \R^2)$ with the following properties: 
\begin{itemize}
	\item[$i)$] $\sup_{\eps>0}\| \nabla' \ffi_{\eps}  \|_{L^{\infty}(U; \R^{2\times 2})} < \frac{1}{\alpha}$;\\[-0.3cm]
	\item[$ii)$] $\ffi_{\eps} $ is constant on $\omega_\eps^k\cap U$ for every $k\in \Z^2$ and $\eps>0$;\\[-0.3cm]
	\item[$iii)$] $\varphi_\eps(\eps(k+ [-\alpha, 1 -\alpha)^2))\cap \varphi_\eps(\eps(j+ [-\alpha, 1 -\alpha)^2))=\emptyset$ for all $k, j\in \Z^2$ with $k\neq j$ and every $\eps>0$;\\[-0.3cm]
	\item[$iv)$] $\ffi_\eps\to {\rm id}_{\R^2}$ uniformly in $U$ as $\eps\to 0$.
\end{itemize}
\end{lemma}
\begin{proof}
We consider the translated unit cell $Z :=Y-\alpha(e_1+e_2)= [-\alpha, 1 -\alpha)^2$ and its partition $Z= \bigcup_{i=1}^4 Z_i$ with
\begin{align*}
	Z_{1} = [-\alpha, \alpha)^2,\  Z_{2} = [-\alpha, \alpha)\times [\alpha, 1-\alpha),\  Z_{3}=[\alpha, 1 -\alpha)^2,\  Z_{4}= [\alpha, 1 -\alpha)\times [-\alpha, \alpha),
\end{align*} 
see Figure \ref{fig:translated_cell}. Define $\ffi: \R^2 \to \R^2$ to be a continuous and piecewise affine function with $Z$-periodic gradients given by 
\begin{align}\label{def_phi}
	\nabla' \ffi = \begin{cases}
							\frac{1}{2\alpha} (e_1 | e_2) & \text{ in } Z_{1},\\
							\frac{1}{2\alpha} (e_1 | 0) & \text{ in } Z_{2},\\
							0 & \text{ in }  Z_{3},\\
							\frac{1}{2\alpha} (0 | e_2) & \text{ in } Z_{4},
					\end{cases}		
\end{align}
that is, up to a translation, 
\begin{align}\label{def_phi_explicit}
	\ffi(k+ z') = k + \begin{cases}
								\frac{1}{2\alpha}z' &\text{ for } z'\in Z_1,\\
								\frac{1}{2\alpha}z_1e_1 + \frac{1}{2}e_2&\text{ for } z'\in Z_2,\\
								\frac{1}{2}(e_1+e_2) &\text{ for } z'\in Z_3,\\
								\frac{1}{2\alpha}z_2e_2 + \frac{1}{2}e_1&\text{ for } z'\in Z_4,
							\end{cases}
\end{align}for $z'=(z_1, z_2)\in Z$ and $k\in \Z^2$.

\begin{figure}[h!]
	\begin{tikzpicture}
		\begin{scope}[scale=3]
			\draw (0,2) rectangle (2,0);
			\draw (1,0) -- (1,2);
			\draw (0,1) -- (2,1);
			
			\draw (0.2,0.8) rectangle (0.8,0.2);
			\draw (1.2,1.8) rectangle (1.8,1.2);
			\draw (0.2,1.8) rectangle (0.8,1.2);
			\draw (1.2,0.8) rectangle (1.8,0.2);
			
			\draw [fill=green!40!white](0.8,1.8) rectangle (1.2,1.2);
			\draw [fill=red!40!white](0.8,1.2) rectangle (1.2,0.8);
			\draw [fill=blue!40!white](1.2,1.2) rectangle (1.8,0.8);
			\draw [fill=orange!40!white](1.2,1.8) rectangle (1.8,1.2);
			
			\draw [fill=gray!30!white,rotate around={45:(0.6,0.4)}](0.5,0.5) ellipse (0.2 and 0.1);
			\draw [fill=gray!30!white,rotate around={-10:(1.5,0.6)}](1.5,0.6) ellipse (0.2 and 0.15);
			\draw [fill=gray!30!white,rotate around={-70:(1.5,1.5)}](1.5,1.5) ellipse (0.24 and 0.12);
			\draw [fill=gray!30!white,rotate around={-30:(0.5,1.35)}](0.5,1.35) ellipse (0.2 and 0.07);
			
			\draw (1,1.5) node {$Z_2$};
			\draw (1,1) node {$Z_1$};
			\draw (1.5,1) node {$Z_4$};
			\draw (1.65,1.7) node {$Z_3$};
			\draw (1.5,1.5) node {$\omega^0_1$};
			
			\draw [<->] (1.5,1.8) --++(0,0.2);
			\draw (1.5,1.9) node [anchor=east] {\footnotesize $\alpha$};
			\draw [<->] (1.8,1.5) --++(0.2,0);
			\draw (1.9,1.5) node [anchor=north] {\footnotesize $\alpha$};
			\draw [<->] (0.8,0.6) --++ (0.4,0);
			\draw (0.925,0.6) node [anchor=north] {\footnotesize $2\alpha$};
			\draw [<->] (0.6,0.8) --++ (0,0.4);
			\draw (0.6,0.94) node [anchor=east] {\footnotesize $2\alpha$};
			\draw [<->] (2.1,0) --++(0,1);
			\draw (2.1,0.5) node [anchor =west] {\small $1$};
			\draw [<->] (1,-0.1) --++(1,0);
			\draw (1.5,-0.1) node [anchor =north] {\small $1$};
			\draw (1.9,1.9) node {$Y$};
		\end{scope}
	\end{tikzpicture}
	\caption{Illustration of the translated unit cell $Z$ and its partition into $Z_1,\ldots,Z_4$; the origin is located in the center of the figure.}\label{fig:translated_cell}
\end{figure} 

Based on the above definitions, we introduce for each $\eps>0$ a Lipschitz function
\begin{align}\label{varphieps}
\varphi_\eps(x') = \eps\varphi\bigl(\tfrac{x'}{\eps}\bigr) + d_\eps
\quad \text{for $x'\in U$,}
\end{align} with the translation $d_\eps\in\R^2$  such that
\begin{align}\label{mean-value-phi_eps}
	\int_{U} \ffi_{\eps}(x') \dd x' = \int_{U} x' \dd x'. 
\end{align}

Since $\nabla' \varphi_\eps=\nabla'  \varphi(\frac{\cdot}{\eps})$ on $U$, we obtain $i)$ immediately from the observation that $|\nabla' \varphi|\leq \frac{\sqrt{2}}{2\alpha}$, and $ii)$ follows since \eqref{unif-compact-condition} translates into
\begin{align*}
	\omega_\eps^k - \eps k \subset \eps Z_3\qquad \text{for all $k\in\Z^2$ and $\eps>0$}
\end{align*}
and $\nabla' \varphi=0$ on $Z_3$. Moreover, it follows from~\eqref{varphieps} and~\eqref{def_phi_explicit} that
\begin{align*}
	\ffi_\eps(\eps(k+Z)) =\eps \varphi(k+Z) +d_\eps =\eps \big( k + \textstyle[-\frac{1}{2},\frac{1}{2})^2\big) +d_\eps,
\end{align*}
for every $\eps>0$ and every $k\in\Z^2$, which implies $iii)$. 

As for the proof of $iv)$, we use the Riemann-Lebesgue lemma on the weak convergence of periodically oscillating sequences in conjunction with \eqref{def_phi} to conclude that
\begin{align}\label{33}
\nabla' \ffi_{\eps} \weaklystar \int_{Z} \nabla'\ffi \dd z' = \frac{1 }{2\alpha} \bigl(|Z_1| (e_1 | e_2)  + |Z_2| (e_1 | 0) + |Z_4| (0 | e_2) \bigr)= {\rm Id}_{\R^{2\times 2}} = \nabla'  {\rm id}_{\R^2}
\end{align}
in $L^\infty(U;  \R^{2 \times 2})$ as $\eps\to 0$. 
In light of \eqref{mean-value-phi_eps}, Poincar\'e's inequality yields for any $q>2$ the existence of a constant $C=C(U, q)$ such that
\begin{align*}
	\left\|  \ffi_{\eps} - \dashint_{ U} x' \dd x' \right\|_{L^q(U; \R^2)} = \left\|  \ffi_{\eps} -\dashint_{U} \ffi_{\eps} \dd x' \right\|_{L^q(U; \R^2)} \le C \| \nabla' \ffi_{\eps} \|_{L^q(U; \R^{2\times 2})} \leq \frac{C}{\alpha},
\end{align*}
with the last estimate due to $i)$. 
Thus, as a uniformly bounded family in $W^{1,q}(U; \R^2)$ satisfying~\eqref{33}, every subsequence of $(\ffi_{\eps})_{\eps}$ has a weakly converging subsequence with limit ${\rm id}_{\R^2} +d$ for some $d \in \R^2$.
By~\eqref{mean-value-phi_eps}, the shift vector $d$ needs to vanish. Thus, we obtain via compact Sobolev embedding that $\ffi_\eps\to {\rm id}_{\R^2}$ uniformly as $\eps\to 0$, as stated.
\end{proof}

First, we present the general idea of how to construct approximating sequences under the simplifying assumption of Lipschitz regularity. This approach serves as a basis for proving the analogous statement for Sobolev functions in Proposition \ref{prop:approx_sobolev}.

\begin{proposition}[Approximation of Lipschitz functions]\label{prop:approx_lipschitz} Let $R\in W^{1,\infty}(\omega; \SO(3))$ and $b\in W^{1,\infty}(\omega; \R^3)$. If $u\in W^{1,\infty}(\Omega;\R^3)$ is given by
	\begin{align*}
		u(x) = R(x')x + b(x'), \quad x\in \Omega,
	\end{align*}
	then there exists a sequence $(u_{\eps})_{\eps} \subset W^{1,\infty}(\Omega; \R^3)$ with 
	\begin{center}
		$\nabla u_{\eps} \in \SO(3)$ a.e. in $\eYrig \cap \Omega$ 
	\end{center} 
	for all $\eps$, such that $u_{\eps} \weaklystar u$ in $W^{1,\infty}(\Omega;\R^3)$ as $\eps\to 0$.
\end{proposition}

\begin{proof}
We start by extending $R\in W^{1,\infty}(\omega; \SO(3))$ according to Lemma \ref{lem:Sobolev-extension}, that is, we consider $R\in W^{1,\infty}(V; SO(3))$, where $V\subset \R^2$ is an open set with $\omega\Subset V$. 
Moreover, one can also find a Lipschitz extension of $b$ in $W^{1, \infty}(V;\R^3)$, still called $b$, by standard Sobolev theory, see e.g., \cite[Theorem 1, Section 3.1]{EvG15}. 

Further, let $(\ffi_{\eps})_{\eps}$ be the sequence of Lipschitz functions that results from approximating the identity on $\R^2$ according to the construction in Lemma~\ref{lem:approx-id} with $U=\omega$. Since $(\ffi_\eps)_\eps$ converges uniformly to $\mathrm{id}_{\R^2}$ on $\omega$ as $\eps\to 0$, we may assume that $\ffi_\eps(\omega)\subset V$ for all $\eps$ sufficiently small.
For such $\eps>0$, we can therefore define
\begin{align}\label{approx_lipschitz}
	u_{\eps}(x) = R(\varphi_{\eps}(x'))x + b(\varphi_{\eps}(x')), \quad x\in \Omega.
\end{align}
Since the composition of two Lipschitz maps is again Lipschitz, one has that $u_{\eps} \in W^{1,\infty}(\Omega; \R^3)$. 
By Lemma~\ref{lem:approx-id}\,$ii)$, $\ffi_{\eps}$ is constant on all the connected components of $(\eYrig)'\cap \omega= \bigcup_{k\in \Z^2}w_\eps^k\cap \omega$, which yields
\begin{align*}
	\nabla u_\eps = R\circ \varphi_\eps \in \SO(3) \text{ a.e.~in $\eYrig\cap \Omega$.}
\end{align*}

It remains to prove that $u_{\eps}\weaklystar u$ in $W^{1,\infty}(\Omega; \R^3)$. 
Indeed, the uniform convergence of $(u_\eps)_\eps$ follows from the Lipschitz continuity of $R$ and $b$ along with the uniform convergence of $(\ffi_\eps)_\eps$ to the identity map on $\R^2$, precisely,
\begin{align*}
	\sup_{x\in\Omega}|u_\eps(x) - &u(x)| = \sup_{x\in\Omega}|(R(\ffi_{\eps}(x')) - R(x'))x + b(\ffi_{\eps}(x')) - b(x')| \\
	&\leq C\sup_{x'\in\omega}\big(|R(\ffi_{\eps}(x')) - R(x')| + |b(\ffi_{\eps}(x')) - b(x')|\big) \\ & \leq C\sup_{x'\in\omega}|\ffi_{\eps}(x') -x'| \to 0 \quad \text{as $\eps\to 0$}.
\end{align*}
Finally, it needs to be shown that $(\nabla u_\eps)_\eps$ is uniformly essentially bounded. We use the chain rule and exploit Lemma~\ref{lem:approx-id}\,$i)$ to obtain the estimate 
\begin{align*}
	\norm{\nabla u_\eps}_{L^\infty(\Omega;\R^3)} &= \essup_{x\in\Omega} \bigl|\nabla R(\ffi_\eps(x'))\nabla \ffi_\eps(x') x + R(\ffi_\eps(x')) + \nabla b(\ffi_{\eps}(x')) \nabla \ffi_\eps(x')\bigr|\\
	&\leq C\essup_{x\in\Omega}\big(1+ |\nabla R(\ffi_\eps(x'))| + |\nabla b(\ffi_\eps(x'))|\big) \\
		&\leq C\big(1+ \norm{R}_{W^{1,\infty}(V;\R^{3\times 3})} + \norm{b}_{W^{1,\infty}(V;\R^3)}\big),
\end{align*}
which concludes the proof.
\end{proof}

Now we can address the analogy of Proposition~\ref{prop:approx_lipschitz} in the setting of $W^{1,p}$-functions.   
The proof strategy is similar but requires a refined reasoning, since the composition of Sobolev with Lipschitz functions may not be Sobolev anymore, see the work by~Conti \& Dolzmann~\cite[Appendix A]{CoD15}.

\begin{proposition}[Approximation of Sobolev functions]\label{prop:approx_sobolev} 
	Let $R\in W^{1,p}(\omega; \SO(3))$ and $b\in W^{1,p}(\omega; \R^3)$ with $p\geq 2$. If $u\in W^{1,p}(\Omega;\R^3)$ is given by
	\begin{align}\label{u_Rb}
		u(x) = R(x')x + b(x'), \quad x\in\Omega,
	\end{align}
	then there exists a sequence $(u_{\eps})_{\eps} \subset W^{1,p}(\Omega; \R^3)$ with 
	\begin{align*}
		\nabla u_{\eps} \in \SO(3) \text{ a.e.~in }\eYrig \cap \Omega
	\end{align*} 
	for every $\eps$, such that $u_{\eps} \weakly u$ in $W^{1, p}(\Omega;\R^3)$ as $\eps\to 0$.
\end{proposition}
\begin{proof}	
Let $u \in W^{1,p}(\Omega;\R^3)$ satisfy~\eqref{u_Rb}. The task is to imitate the construction in~\eqref{approx_lipschitz}, while making sure that the approximating sequence actually lies in $W^{1,p}(\Omega;\R^3)$.

To this end, we first approximate $R$ and $b$ in $W^{1,p}$ by smooth functions $R_\eta$ and $b_\eta$ with $\eta>0$ small. However, the natural approach of choosing $u_{\eta, \eps}$ according to~\eqref{approx_lipschitz} for each $\eta$ and concluding by a diagonalization argument, is not easily accessible. This is because the required uniform $L^p$-bounds on $\nabla u_{\eta, \eps}$ are not trivial to obtain due to the lack of invertibility of the approximation of the identity $\varphi_\eps$ (cf.~Lemma~\ref{lem:approx-id}), which prevents a classical change of variables. 
To overcome this issue, we proceed similarly to~\cite[Lemma A.1]{CoD15} and refine the definition of $u_{\eta,\eps}$ by introducing small translations of the independent variables.
	
We detail these arguments in the following three steps.  \medskip

\textit{Step 1: Extension and approximation.}
In analogy to Proposition~\ref{prop:approx_lipschitz}, we first extend $R$ and $b$ to $R\in W^{1,p}(V;\SO(3))$, $b\in W^{1,p}(V;\R^3)$ with an open set $V\subset \R^2$ such that $\omega\Subset V$.	
Moreover, let $U\subset \R^2$ be a bounded Lipschitz domain with $\omega\Subset U\Subset V$ and $(\varphi_\eps)_\eps$ the approximation of the identity from Lemma~\ref{lem:approx-id}. Suppose in the following that $\eps>0$ is so small that $\ffi_\eps(U)+B(0,\eps)\subset V$.
	
According to Lemma \ref{lem:approx_SO(3)} and standard Sobolev theory, there are approximating sequences $(R_\eta)_\eta \subset C^\infty(V;\SO(3))$ and $(b_\eta)_\eta\subset C^\infty(V;\R^3)$ such that
	\begin{align}\label{approx_R,b}
		R_\eta \to R \text{ in } W^{1,p}(V;\SO(3))\quad \text{and}\quad b_\eta \to b \text{ in } W^{1,p}(V;\R^3),
	\end{align}
	respectively; 
notice that we may assume without loss of generality that $(R_\eta)_\eta \subset C^\infty(\overline{V};\SO(3))$ and $(b_\eta)_\eta\subset C^\infty(\overline{V};\R^3)$, since otherwise, we introduce an intermediate set $\tilde V$ with $U\Subset \tilde V\Subset V$ and choose $\eps$ even smaller to guarantee that $\ffi_\eps(U)+B(0,\eps)\subset \tilde V$.\medskip

	\textit{Step 2: Construction of the approximating sequence.}
	Similarly to \eqref{approx_lipschitz}, we define for $\eta>0$, $\eps>0$ sufficiently small, and $a\in B(0,\eps)\subset \R^2$ the Lipschitz functions
	\begin{align}\label{ua}
		u^a_{\eta,\eps}(x) = R_\eta(\ffi_\eps(x') + a)x + b_\eta(\ffi_\eps(x')+a),\quad x\in U\times (0,L).
	\end{align}
Then, 
	\begin{align}\label{stiff_fibers}
		\nabla u_{\eta,\eps}^a \in \SO(3)\text{ a.e.~in } \eYrig\cap U,
	\end{align}
	since $\ffi_\eps$ is constant on $\omega_\eps^k\cap U$ for each $k\in \Z^2$ according to 
	Lemma~\ref{lem:approx-id}\,$ii)$.
	Further, we infer from Lemma~\ref{lem:approx-id}\,$iv)$ that any sequence $(u_{\eta,\eps}^{a_\eps})_\eps$ with $a_\eps\in B(0,\eps)$ converges uniformly for $\eps\to 0$ to a limit function $u_\eta$ given by 
	\begin{align*}
		u_\eta(x) = R_\eta(x')x + b_\eta(x'),\quad  x\in U\times (0, L); 
	\end{align*}
 	in particular, 
	\begin{align*}
		u_{\eta,\eps}^{a_\eps} \to u_\eta\ \text{ in $L^p(U\times (0,L);\R^3)$} \quad \text{as $\eps\to 0$. }
	\end{align*}	
	Since also $u_\eta\to u$ in $W^{1,p}(U\times (0,L);\R^3)$ for $\eta\to 0$ due to~\eqref{approx_R,b}, a diagonalization argument in the sense of Attouch provides a diagonal sequence $(u_\eps)_\eps = (u_{\eta(\eps),\eps}^{a_\eps})_\eps$ such that
	\begin{align*}
		u_\eps \to u\text{ in }\ L^p(U\times (0,L);\R^3). 
	\end{align*}
	In view of \eqref{stiff_fibers}, we know that $\nabla u_\eps \in \SO(3)$ almost everywhere in $\eYrig\cap U$. \medskip

	\textit{Step 3: Choice of suitable translations.} It remains to show that the construction of sequences $(u_\eps)_\eps$ in Step~2 gives rise to a sequence that converges weakly in $W^{1,p}(U;\R^3)$.
This follows immediately, if we can select translations $a_\eps\in B(0,\eps)\subset \R^2$ such that
	\begin{align}\label{444}
		\norm{\nabla u_{\eta,\eps}^{a_\eps}}_{L^p(\Omega;\R^{3\times 3})} \leq C
	\end{align}
	with a constant $C>0$ independent of $\eta$ and $\eps$.
	
To this end, consider the index set $I_\eps = \{k\in \Z^2: |\eps(k+Z)\cap \omega|>0\}$, recalling $\Omega=\omega\times (0, L)$ and the notation $Z=[-\alpha, 1-\alpha)^2$ from the proof of Lemma~\ref{lem:approx-id}, and observe that 
	\begin{align}\label{covering}
		\omega \subset \bigcup_{k\in I_\eps} \eps(k+Z) \subset U
	\end{align}
	for $\eps$ sufficiently small. For such $\eps$, $k\in I_\eps$ and $\eta>0$, we
 define the function
	\begin{align*}
		h_{\eta,\eps}^{k}: B(0,\eps)\to [0,\infty),\: a \mapsto \norm{\nabla u_{\eta,\eps}^{a}}_{L^p(\eps(k+Z)\times(0,L);\R^{3\times 3})}^p,
	\end{align*}
	with $u_{\eta, \eps}^a$ as in~\eqref{ua}. For the mean value of $h_{\eta, \eps}^k$, we obtain from the chain and product rule, together with Lemma~\ref{lem:approx-id}\,$i)$, Fubini's theorem, and a change of variables that
	\begin{align}\label{344}
	\begin{split}
		\dashint_{B(0,\eps)} h_{\eta,\eps}^k(a) \dd a &= \dashint_{B(0,\eps)}\int_0^L\int_{\eps(k+Z)} |\nabla u_{\eta,\eps}^a|^p\dd x'\dd x_3\dd a\\
		&\leq C \int_{\eps(k+Z)}\dashint_{B(0,\eps)} |\nabla' R_\eta(\ffi_\eps(x') +a)|^p + |\nabla' b_\eta(\ffi_\eps(x') +a)|^p + 1 \dd a \dd x'\\
		&\leq C\frac{|\eps(k+Z)|}{|B(0,\eps)|}\int_{B(0,\eps) + \ffi_\eps(\eps(k+Z))} |\nabla' R_\eta(a)|^p + |\nabla' b_\eta(a)|^p +1 \dd a\\
		&\leq C\int_{B(0,\eps) + \ffi_\eps(\eps(k+Z))} |\nabla' R_\eta(a)|^p + |\nabla' b_\eta(a)|^p +1 \dd a.
		\end{split}
	\end{align}
	where $C>0$ is a constant that depends only on $p$ and $\alpha$.
	
	By Lemma~\ref{lem:approx-id}\,$iii)$, the sets $\ffi_\eps(\eps(k+Z))$ with different $k\in \Z^2$ are disjoint for every $\eps>0$.
	Therefore,  every $x'\in \ffi_\eps(U) + B(0,\eps)$ is contained in at most $9$ sets of the form $\ffi_\eps(\eps(k+Z)) + B(0,\eps)$ with $k\in \Z^2$, see also Figure \ref{fig:overlapping}. 	
	\begin{figure}[h!]
		\centering
		\begin{tikzpicture}
			\begin{scope}[scale=1.3]
				\draw [fill=orange!10!white] (0,1) --++ (0,1) -- ([shift={(180:1)}]1,2) arc (180:90:1) --++ (1,0) -- ([shift={(90:1)}]2,2) arc (90:0:1) --++(0,-1) -- ([shift={(0:1)}]2,1) arc (0:-90:1) --++ (-1,0) -- ([shift={(-90:1)}]1,1) arc (-90:-180:1);
			
				\foreach \x in {0,...,3}
					{\draw (0,\x) -- (3,\x);
					 \draw [dashed] (-0.5,\x) -- (0,\x);
					 \draw [dashed] (3,\x) -- (3.5,\x);}
				\foreach \x in {0,...,3}
					{\draw (\x,0) -- (\x,3);
					 \draw [dashed] (\x,-0.5) -- (\x,0);
					 \draw [dashed] (\x,3) -- (\x,3.5);}
					 
				\fill [red!40!white] (1,1)--(2,1)--(2,2)--(1,2)--cycle;
				\draw [green!70!black,ultra thick] (1,2)--(2,2);
				\draw [blue!60!white,ultra thick] (2,1)--(2,2);
				\fill (2,2) [orange!70!white] circle (1.2pt);
				\draw (1.5,0) node [anchor=north] {\small $\ffi_\eps(\eps(k+Z)) + B(0,\eps)$};
				\draw (1.45,1) node [anchor=north] {\small $\ffi_\eps(\eps(k+Z))$};
				
				\draw (5,2.25) node [red!70!white] {\small $\ffi_\eps(\eps(k+Z_1))$};
				\draw (5,1.75) node [green!60!black] {\small $\ffi_\eps(\eps(k+Z_2))$};
				\draw (5,1.25) node [blue!70!white] {\small $\ffi_\eps(\eps(k+Z_3))$};
				\draw (5,0.75) node [orange!90!white] {\small $\ffi_\eps(\eps(k+Z_4))$};
			\end{scope}
		\end{tikzpicture}
		\caption{Illustration of $\ffi_\eps(\eps(k+Z))$, its components $\ffi_\eps(\eps(k+Z_k))$ for $k=1,\ldots,4$, as well as its enlarged version $\ffi_\eps(\eps(k+Z))+B(0,\eps)$ for some $k\in \Z^2$ and $\eps>0$.}\label{fig:overlapping}
	\end{figure}		
	Summing over all $k\in I_\eps$ in~\eqref{344} then implies in view of \eqref{covering} and the choice of $V$ that 
	\begin{align*}
		\dashint_{B(0,\eps)}\norm{\nabla u_{\eta,\eps}^a}_{L^p(\Omega;\R^3)}^p \dd a &\leq C\int_{V} |\nabla' R_\eta(a)|^p + |\nabla' b_\eta(a)|^p + 1 \dd a \\
	&\leq C\bigl(\norm{R_\eta}^p_{W^{1,p}(V;\R^{3\times 3})} + \norm{b_\eta}^p_{W^{1,p}(V;\R^3)}+1\bigr),
	\end{align*}
	and we conclude, due to \eqref{approx_R,b}, that
	\begin{align}\label{bounded_average}
		\dashint_{B(0,\eps)}\norm{\nabla u_{\eta,\eps}^a}_{L^p(\Omega;\R^3)}^p \dd a &\leq C
	\end{align}
	with a constant $C>0$ independent of $\eta$ and $\eps$.
	Hence, there exists for each $\eps$ a subset $E_\eps\subset B(0,\eps)$ of positive measure such that
	\begin{align*}
		\norm{\nabla u_{\eta,\eps}^{a}}_{L^p(\Omega;\R^3)}^p &\leq C
	\end{align*}
	for all $a\in E_\eps$ with the same constant as in \eqref{bounded_average}, so that choosing any $a_\eps\in E_\eps\subset B(0,\eps)$ yields the desired bound~\eqref{444}. This finishes the proof. 
\end{proof}

The proof of Theorem \ref{theo:characterization} is essentially a consequence of Propositions~\ref{prop:approx_sobolev} and~\ref{prop:rigidity_general}. The gap in the different representations can be closed by a lifting argument.  

\begin{proof}[Proof of Theorem~\ref{theo:characterization}]\label{rem:representation}
	It follows from Proposition \ref{prop:rigidity_general} that any $u\in \Acal_0$ (see~\eqref{A_0}) can be represented as
	\begin{align}\label{Sigma_d_repr}
		u(x) = x_3\Sigma(x') + d(x'), \quad x\in\Omega,
	\end{align}
	with $\Sigma\in W^{1,p}(\omega;\Scal^2)$ and $d\in W^{1,p}(\omega;\R^3)$.
	On the other hand, Proposition~\ref{prop:approx_sobolev} yields that any function
	\begin{align}\label{R_b_repr}
		u(x) = R(x')x + b(x'), \quad  x\in\Omega,
	\end{align}
	with $R\in W^{1,p}(\omega;\SO(3))$ and $b\in W^{1,p}(\omega;\R^3)$ lies in $\Acal_0$. If the set of all functions of the form~\eqref{Sigma_d_repr} and~\eqref{R_b_repr} are denoted by $\Acal_{\Sigma}$ and $\Acal_{R}$, respectively, we have
	\begin{align*}
	\Acal_R \subset \Acal_0\subset \Acal_\Sigma.
	\end{align*}

To conclude the proof, it suffices to show that the sets $\Acal_R$ and $\Acal_\Sigma$ coincide under the assumption (H). While $\Acal_R\subset \Acal_\Sigma$ is immediately clear, the converse inclusion is more delicate as it requires the construction of an $\SO(3)$-valued Sobolev function whose third column coincides with $\Sigma$. 
By the lifting result in Lemma \ref{lem:lifting}, we find for any $u$ as in~\eqref{Sigma_d_repr} a function $R\in W^{1,p}(\omega;\SO(3))$ such that $Re_3=\Sigma$, and hence, setting
	\begin{align*}
		b(x') = d(x') - x_1R(x')e_1 - x_2R(x')e_2, \quad x'\in\omega,
	\end{align*}
	implies~\eqref{R_b_repr}. 	
\end{proof}

Even though the representation of functions in the form~\eqref{R_b_repr} is essential for our construction of an approximating sequence in Propositions~\ref{prop:approx_lipschitz} and~\ref{prop:approx_sobolev}, we point out that the choice of $R$ is not unique, cf.~Remark~\ref{rem:lifting_formula}. 
More intuitive from a geometric point of view is~\eqref{Sigma_d_repr}; indeed, one can describe the image of $\Omega$ under $u$ as the image of the cross section $\omega$ under the map $d$, fattened linearly in $x_3$ in the direction of the vector field $\Sigma$. 
A few illustrative examples are presented in the next section.

\section{Examples of effective deformations}\label{sec:examples}
Before we focus on explicit examples of limit functions as characterized in Theorem~\ref{theo:characterization}, which  describe the macroscopically attainable deformations with rigid fiber-reinforcments, let us briefly address the issue of incompressibility, or in other words, local volume preservation, of such maps. 

Throughout this section,we  consider $u\in \Acal_0\subset W^{1, p}(\Omega;\R^3)$ with $p>2$ (see \eqref{A_0})
of the form 
\begin{align}\label{555}
	u(x) = x_3\Sigma(x') + d(x'),\quad x\in\Omega,
\end{align}
with given $\Sigma\in W^{1,p}(\omega;\Scal^2)$ and $d\in W^{1,p}(\omega;\R^3)$, cf.~also Remark \ref{rem:representation}.

The next lemma gives a necessary condition for the incompressibility of such deformations.
\begin{lemma}\label{lem:parallel}
Let $u$ as in~\eqref{555} be incompressible, i.e., $\det \nabla u = 1$ a.e.~in $\Omega$. Then,
\begin{align}\label{parallel}
	\partial_1 \Sigma \parallel \partial_2 \Sigma\text{ a.e.~in } \omega.
\end{align}
\end{lemma}
\begin{proof} 
In fact, the condition~\eqref{parallel} results from a second-order linearization in $x_3$-direction of the incompressibility constraint, $\det \nabla u=1$ a.e.~in $\Omega$.
By the multi-linearity of the determinant, it follows for a.e.~ $x=(x',x_3)\in \Omega$ that
\begin{align*}
	0= \partial_3^2\det \nabla u(x) &= \partial_3^2 \det \big(\partial_1 d(x') + x_3\partial_1\Sigma(x')|\partial_2 d(x') + x_3\partial_2\Sigma(x')| \Sigma(x') \big) \\
	&= 2 \det \bigl(\partial_1\Sigma(x')| \partial_2 \Sigma(x')| \Sigma(x')\bigr).
\end{align*}
Thus, $\partial_1\Sigma, \partial_2\Sigma$ and $\Sigma$ are linearly dependent a.e.~in $\omega$, that is, for a.e.~$x'\in \omega$, there exists $\lambda=(\lambda_1, \lambda_2, \lambda_3)\in \R^3\setminus\{0\}$ such that
\begin{align}\label{556}
	\lambda_1\partial_1\Sigma + \lambda_2\partial_2\Sigma + \lambda_3\Sigma = 0. 
\end{align} 
Since $|\Sigma|=1$ and $\partial_i\Sigma\cdot\Sigma=0$ a.e.~in $\omega$ for $i\in\{1,2\}$, it follows from scalar multiplication of~\eqref{556} with $\Sigma$ that $\lambda_3=0$, which shows~\eqref{parallel}. 
\end{proof}

In the following, we provide a few illustrative examples of (compressible and incompressible) deformations of the form \eqref{555}, where $\Omega$ is always a suitable open cuboid.

\begin{example}[\boldmath{$\Sigma$} is constant]
If $\Sigma$ is constant, the image $u(\Omega)$ corresponds to the deformed cross section $d(\omega)$ thickened in the direction of $\Sigma$ by the hight of $\Omega$.
A first example of an incompressible deformation of this type is
	\begin{align}\label{paraboloid}
		u(x) = x_3\begin{pmatrix}0\\0\\1\end{pmatrix} + \begin{pmatrix}x_1\\x_2\\-x_1^2-x_2^2\end{pmatrix},\quad x\in\Omega,
	\end{align}
	where $\omega$ is transformed into (parts of) a paraboloid, see~Figure \ref{fig:sigma_constant}a). 
	Another classical example in this context is a simple shear in $e_2$-direction, i.e.,
	\begin{align}\label{shear}
		u(x) = x_3\begin{pmatrix}0\\0\\1\end{pmatrix} + \begin{pmatrix}x_1 \\ \gamma x_1 +x_2 \\ 0\end{pmatrix},\quad x\in\Omega,
	\end{align}
with shear parameter $\gamma\in\R$, see Figure \ref{fig:sigma_constant}b).

	\begin{figure}[h!]
		\centering
		\begin{subfigure}{.49\linewidth}
			\centering
			\includegraphics[height=5cm]{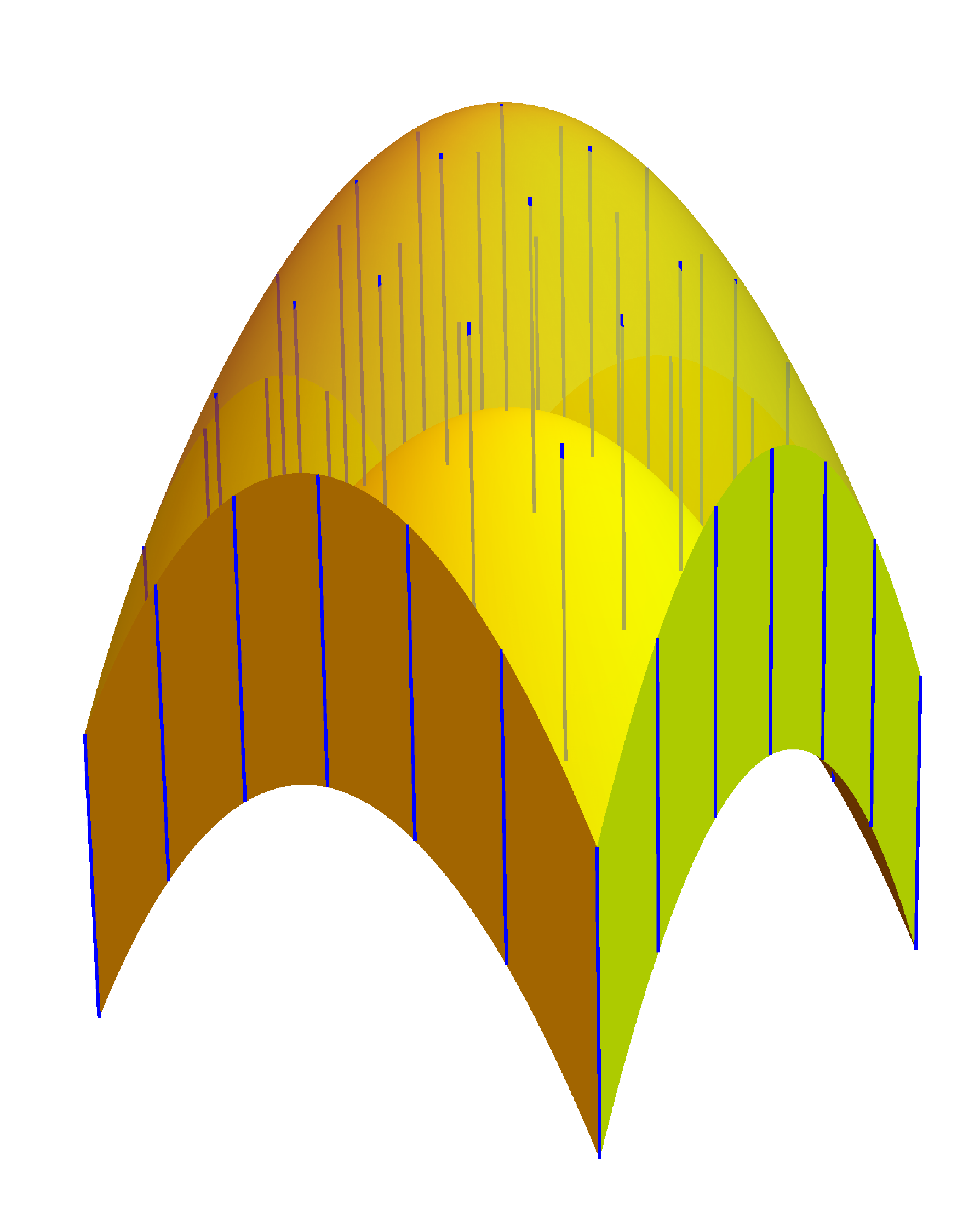}
			\put (-120,120) {a)}
		\end{subfigure}
		\begin{subfigure}{.49\linewidth}
			\centering
			\includegraphics[height=5cm]{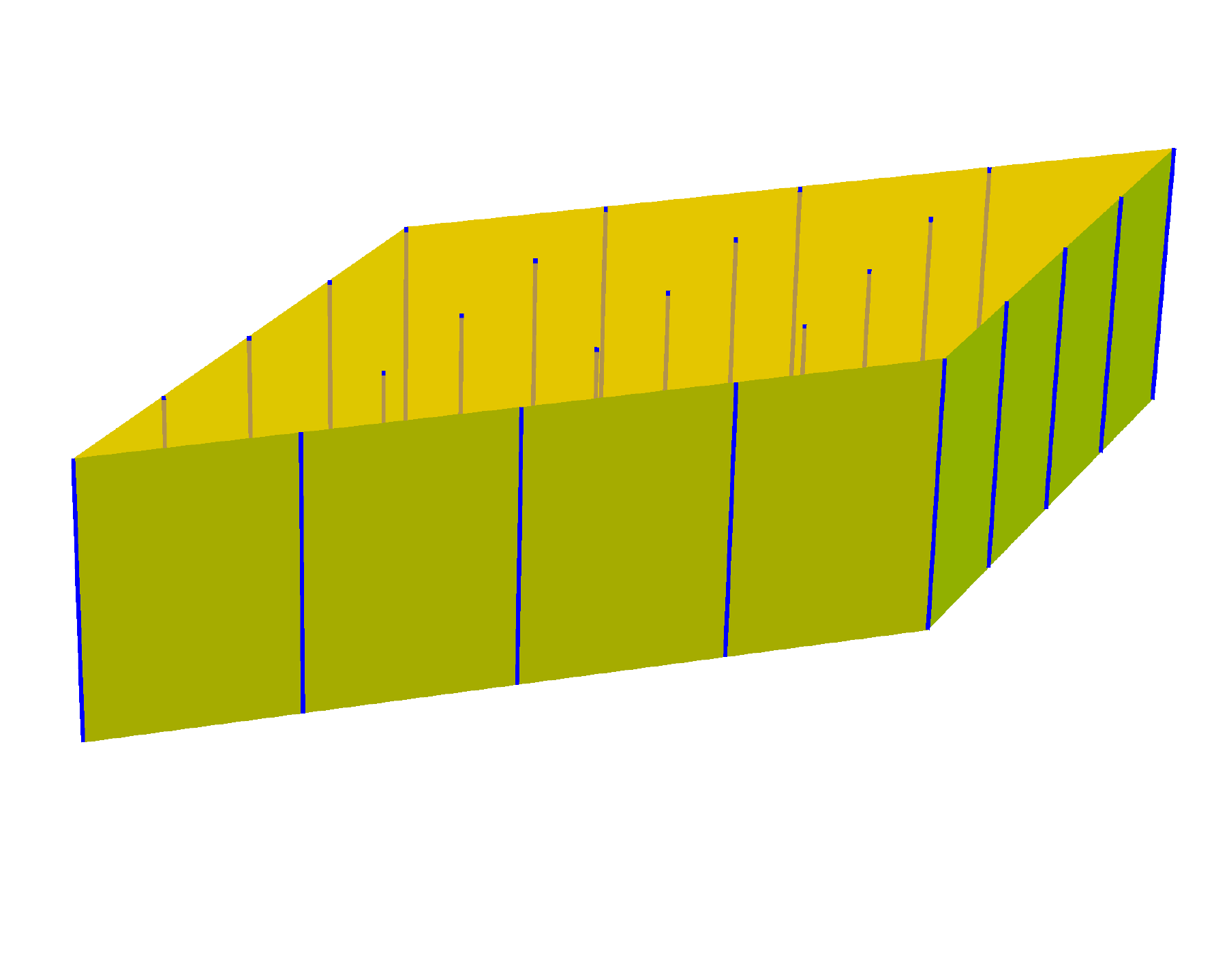}
			\put (-160,120) {b)}
		\end{subfigure}
		\caption{The images of $\Omega=(-1,1)^2\times (0,1)$ under a) \eqref{paraboloid} and  b) \eqref{shear} for $\gamma=1$.}\label{fig:sigma_constant}
	\end{figure}
\end{example}

\begin{example}[\boldmath{$\Sigma$} depends on only one variable]
Suppose that $\partial_2 \Sigma=0$. In this case, the direction along which the transformed cross section $d(\omega)$ is thickened depends in general non-trivially on $x_1$. 
	A locally volume-preserving deformation describing a twist in $e_1$-direction is given by 
	\begin{align}\label{twist}
		u(x)= x_3\begin{pmatrix} 0 \\ - \sin(\tfrac{x_1}{\pi}) \\[0.1cm] \cos(\tfrac{x_1}{\pi})\end{pmatrix} 
					+\begin{pmatrix} x_1\\ x_2 \cos(\tfrac{x_1}{\pi})\\[0.1cm] x_2\sin(\tfrac{x_1}{\pi})\end{pmatrix}, \quad x\in \Omega,
	\end{align}
	see~Figure \ref{fig:sigma_one}a).
	As a second example, consider
	\begin{align}\label{tyre}
		u(x) = \frac{x_3}{r}\begin{pmatrix}x_1\\0\\ \sqrt{r^2 - x_1^2}\end{pmatrix} + \begin{pmatrix}x_1\\x_2+e^{x_2}\\\sqrt{r^2-x_1^2}\end{pmatrix},\quad x\in\Omega,
	\end{align}
with $r>0$. Notice that $u$ as in~\eqref{tyre} is not incompressible despite satisfying the necessary condition \eqref{parallel}, since $u$ involves a stretch in $e_2$-direction, see Figure \ref{fig:sigma_one}b). 

	\begin{figure}[h!]
		\centering
		\begin{subfigure}{.495\linewidth}
			\centering
			\includegraphics[height=4.4cm]{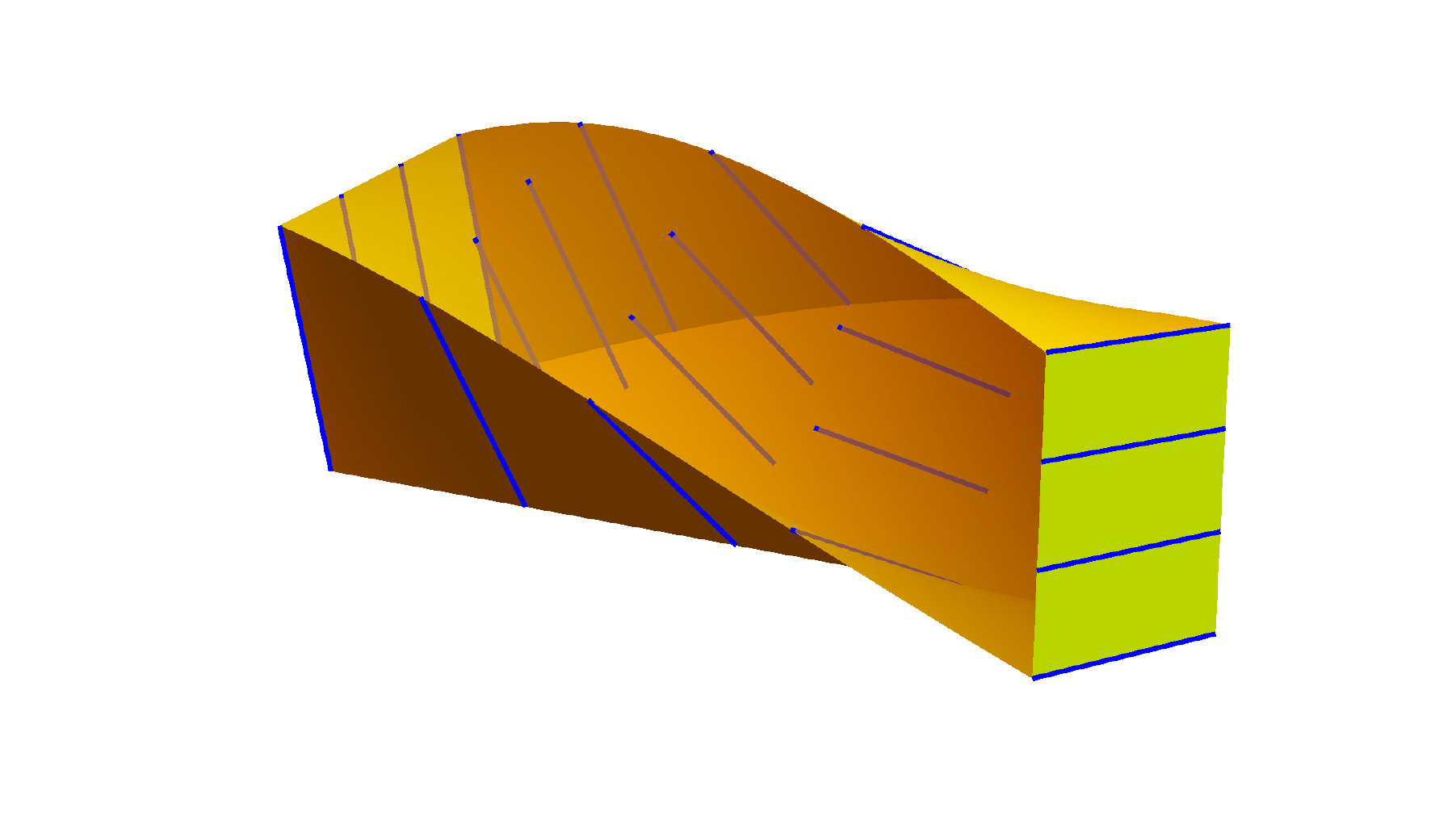}
			\put (-180,120) {a)}
		\end{subfigure}
		\begin{subfigure}{.495\linewidth}
			\centering
			\includegraphics[height=4.4cm]{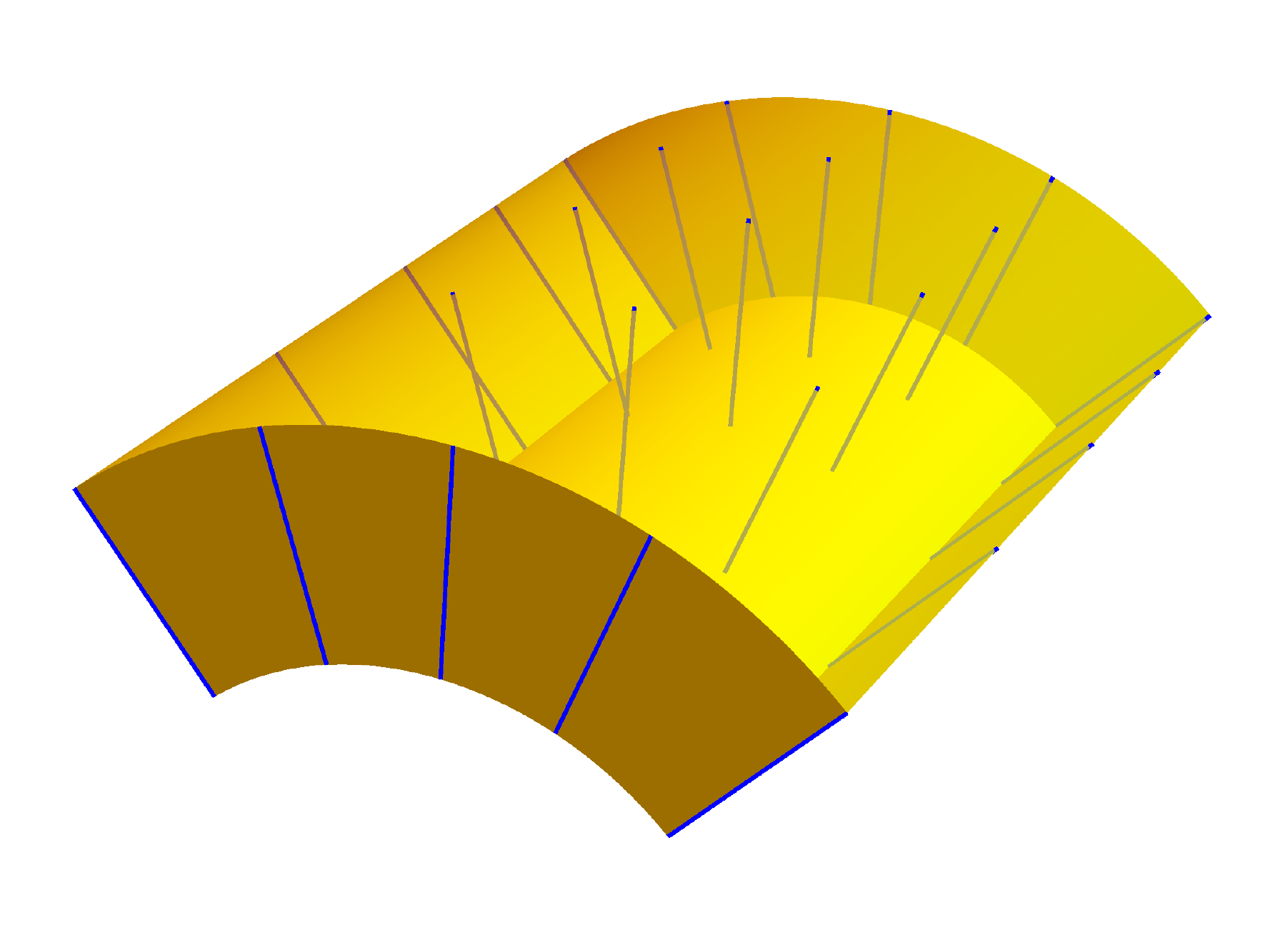}
			\put (-150,120) {b)}
		\end{subfigure}
		\caption{a) The image of $\Omega=(0,4)\times (0,1)^2$ under \eqref{twist}. b) The image of $\Omega=(-1,1)^2\times (0,1)$ under \eqref{tyre} with $r=\frac{3}{2}$.}\label{fig:sigma_one}
	\end{figure}
\end{example}

\begin{example}[\boldmath{$\Sigma$} depends on both cross-section variables] \label{ex:two_variables}
	Consider the following modification of \eqref{paraboloid},
	\begin{align}\label{hedgehog}
		u(x) = \frac{x_3}{\sqrt{4(x_1^2+x_2^2)+1}}\begin{pmatrix}
		2x_1\\ 2x_2 \\ 1
		\end{pmatrix} + \begin{pmatrix}
			x_1 \\ x_2 \\ -x_1^2 - x_2^2
		\end{pmatrix},\quad x\in\Omega.
	\end{align}
	In this case, the vector field $\Sigma$ is orthogonal to the surface of the paraboloid $x'\mapsto (x_1,x_2,-x_1^2-x_2^2)$, see Figure \ref{fig:sigma_two}a).
	Since $\partial_1 \Sigma$ is not parallel to $\partial_2 \Sigma$, $u$ as in~\eqref{hedgehog} is not locally volume-preserving according to Lemma~\ref{lem:parallel}.
	Another example that does not satisfy the condition~\eqref{parallel} either is
	\begin{align}\label{trophy}
		u(x) = \frac{x_3}{\sqrt{2}\sqrt{x_1^2+x_2^2 +1}}\begin{pmatrix}-x_1-x_2\\x_1-x_2\\\sqrt{2}\end{pmatrix} +\begin{pmatrix}x_1\\x_2\\0\end{pmatrix}, \quad x\in \Omega,	
	\end{align}
depicted in Figure \ref{fig:sigma_two}b).
	\begin{figure}[h!]
		\centering
		\begin{subfigure}{.49\linewidth}
			\centering
			\includegraphics[height=5.5cm]{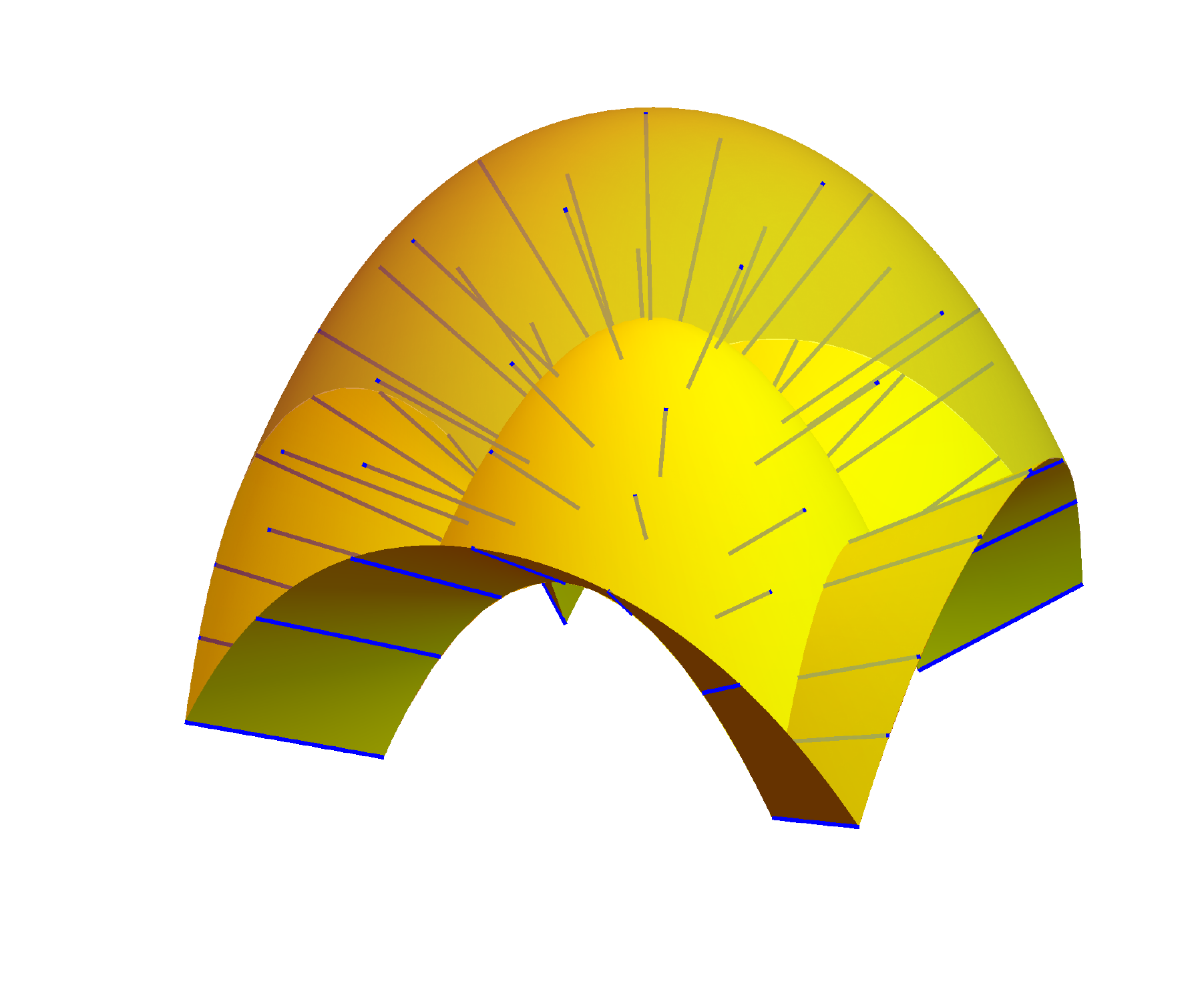}
			\put (-160,140) {a)}
		\end{subfigure}
		\begin{subfigure}{.49\linewidth}
			\centering
			\includegraphics[height=5.5cm]{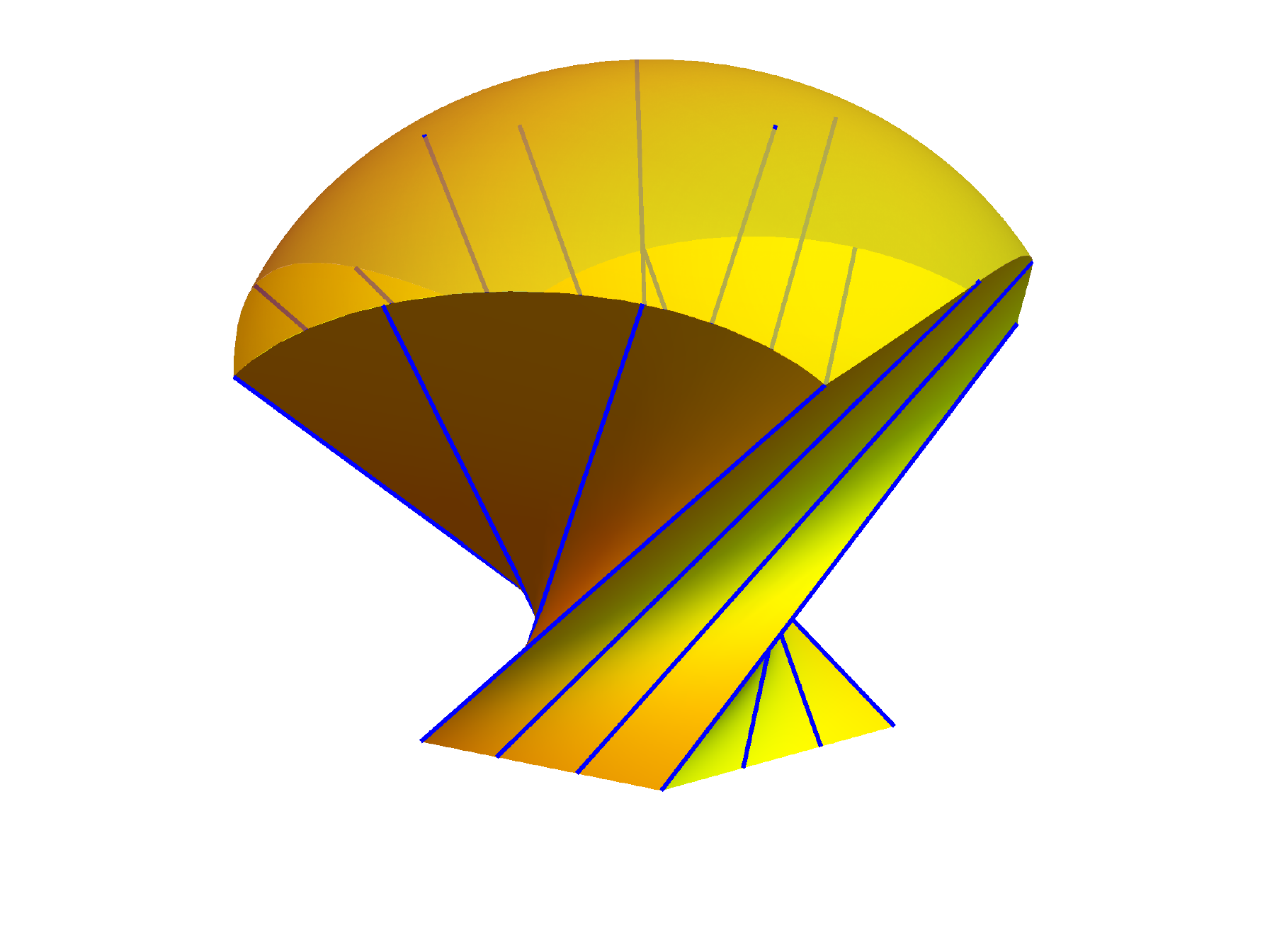}
			\put (-180,140) {b)}
		\end{subfigure}
		\caption{a) The image of $\Omega=(-1,1)^2\times (0,1)$ under \eqref{hedgehog}. b) The image of $\Omega=(-1,1)^2\times (0,4)$ under \eqref{trophy}.}\label{fig:sigma_two}
	\end{figure}
\end{example}

\begin{remark}[Comparison with layered composites]\label{rem:comparison}
	Let $\Omega= (0,L_1)\times (0,L_2)\times (0,L_3)$ be an open cuboid. 
	As proven in~\cite[Theorem~1.1]{ChK20}, the admissible effective deformations of a composite with rigid layers, or mathematically speaking, the weak $W^{1,p}$-limits of sequences $(u_\eps)_\eps$ such that
	\begin{align*}
		u_\eps\in \Bcal_\eps := \{u\in W^{1,p}(\Omega;\R^3) : \nabla u \in \SO(3) \text{ a.e.~in } X_\eps^{\rm rig} \cap \Omega\}
	\end{align*}
	for $\eps>0$, with $X_\eps^{\rm rig}$ as introduced in \eqref{layers}, are characterized by
	\begin{align*}
		\Bcal_0 := \{u\in W^{1,p}(\Omega;\R^3) : &\ u(x) = R(x_1)x+b(x_1) \text{ for a.e.~$x\in\Omega$} \\ 
		&\text{ with $R\in W^{1,p}((0,L_1);\SO(3)), b\in W^{1,p}((0,L_1);\R^3)$} \}.
	\end{align*}
	Comparing with~the characterization of $\Acal_0$ in Theorem~\ref{theo:characterization} shows that $\Bcal_0\subset \Acal_0$.  This observation backs the intuition that composites with rigid fiber reinforcements are more flexible in their deformation behavior than those with rigid layers. 
		
	It is evident that the macroscopic deformations in \eqref{shear} and \eqref{twist} can be attained also by layered materials, meaning that they are elements of $\Bcal_0$. 
	On the other hand, the deformations presented in Example \ref{ex:two_variables}, as well as \eqref{paraboloid} and \eqref{tyre} show that the inclusion $\Bcal_0\subset\Acal_0$ is strict, which underlines the higher flexibility fiber-reinforced materials.
\end{remark}

\section{Regularization in the cross-section variables}\label{sec:regularization}

This section is concerned with the proof of Theorem \ref{theo:rigidity_reg}, where we additionally assume that second derivatives in the cross-section variables  
of weakly convergent sequences $(u_\eps)_\eps$ with $u_\eps\in\Acal_\eps$ exist, and are $L^p$-bounded  uniformly in $\eps$. 
In this case, the weak $W^{1,p}$-limit of $(u_\eps)_\eps$ corresponds to a rigid body motion. 
We begin our analysis with two auxiliary results.

First, the above-mentioned assumption of higher regularity allows us to improve the estimates in Lemma \ref{lem:neighboring_rotations} with the help of a one-dimensional Poincar{\'e} estimate. 
\begin{lemma}\label{lem:neighboring_rotations_reg}
	Let $E\subset\R^3$, $d\in\R^2$ and $w_1, w_2: E\to \R^3$ be as in Lemma \ref{lem:neighboring_rotations}. 
If $v\in W^{1,p}(E;\R^3)$ with $p\geq 1$ satisfies  \begin{align*}
	\max_{i, j\in \{1,2\}}\norm{\partial_i\partial_j v}_{L^p(E;\R^3)}^p  <\infty 
	\end{align*} 
	and if
	\begin{align*}
		v = w_{1} \text{ a.e.~in } E \cap \big((0,\mu)\times \R^2\big)\quad\text{and}\quad v = w_{2} \text{ a.e.~in } E \cap \big((L_1-\mu,L_1)\times \R^2\big)
	\end{align*}
	for some $\mu>0$, then
	\begin{align}\label{est61}
		\int_{E}|\partial_d^2 v|^p \dd{x} \geq \frac{C|E|L_3^p}{(1+|m|^p)^{2} L_1^{2p}}|(A_2-A_1)e_3|^p
	\end{align}
	with a constant $C>0$ that depends only on $p$. 
\end{lemma}
\begin{proof}
	We may assume without loss of generality that $w_1=0$, otherwise consider $\tilde{v}=v-w_1$ in place of $v$.
	Then, by assumption, $\partial_d v=0$ a.e.~in $E \cap \big((0,\mu)\times \R^2\big)$, and thus, with $u(y)=v(y_1,y_2+my_1,y_3)$ for $y\in Q=(0,L_1)\times(0,L_2)\times (0,L_3)$,
	\begin{align*}
		\partial_1 u=0\quad \text{ a.e.~in $(0, \mu)\times (0, L_2)\times (0, L_3)$.}
	\end{align*}
	
	Via the same change of variables as in Lemma~\ref{lem:neighboring_rotations},
	we can therefore deduce with Poincar\'e's inequality, applied to $\partial_1 u$ in $y_1$-direction, that
	\begin{align*}
		\begin{split}
			\int_E |\partial_d^2 v|^p \dd x &= \frac{1}{(1+m^2)^p} \int_{0}^{L^3}\int_0^{L^2} \int_{0}^{L_1}|\partial_1(\partial_1 u)|^p \dd y_1 \dd y_2\dd y_3  \\
			& \geq \frac{C}{(1+m^2)^{p}L_1^{p}} \int_Q |\partial_1 u|^p \dd y = \frac{C}{(1+m^2)^{p/2}L_1^{p}} \int_E |\partial_d v|^p \dd x,
		\end{split}
	\end{align*}
	where $C^{-1}L_1^{p}$ with $C>0$ depending only on $p$ is the optimal Poincar{\'e} constant. 
	In combination with Lemma \ref{lem:neighboring_rotations}, which implies
	\begin{align*}
		\int_E |\partial_d v|^p \dd x \geq \frac{C|E|L_3^p}{(1+|m|^p)L_1^{p}}|A_2e_3|^p,
	\end{align*}
	we obtain the desired estimate.
\end{proof}

Second, we prove that strongly $L^p$-convergent functions that are constant rotations on the rigid components have a constant limit in the set of rotations.
In particular, when applied to gradient fields, the next result can be seen as a companion to Proposition \ref{prop:rigidity_general} for strongly converging sequences in $W^{1,p}(\Omega;\R^3)$. 

\begin{lemma}\label{lem:vector_field}
	Let $U\subset \R^3$ be an open set and let $(V_\eps)_\eps\subset L^p(U;\R^{3\times 3})$ with $p\geq 1$ satisfy
	\begin{align*}
		V_\eps \in \SO(3) \text{ a.e.~in } \eYrig\cap U
	\end{align*}	
	for all $\eps$.
	If $V\in L^p(U;\R^{3\times 3})$ is such that  $V_\eps \to V$ in $L^p(U;\R^{3\times 3})$ as $\eps\to 0$, then 
	\begin{align*}
		V\in \SO(3) \text{ a.e.~in } U.
	\end{align*}
\end{lemma}
\begin{proof}
	Assume to the contrary that there is a $\gamma>0$ and an open cube $Q\subset U$ such that 
	\begin{align*}
		\dist(V(x), \SO(3)) >\gamma \quad\text{ for a.e.~$x\in Q$. }
	\end{align*}
	Since 
	\begin{align*}
		\gamma < \dist(V(x),\SO(3)) \leq |V(x) - V_\eps(x)|\quad \text{ for a.e.~$x\in \eYrig\cap Q$,}
	\end{align*}
	it follows that, up to a set of measure zero,
	\begin{align*}
		\eYrig\cap Q \subset \{x \in Q: |V_\eps(x) - V(x)| > \gamma\}. 
	\end{align*} 
	Recalling the definition of $Y_{\eps}^{\rm rig}$,  each fiber cross-section $w_\eps^k$ with $\eps>0$ and $k\in \Z^2$ contains a square $S_\eps^k$ with $|S_\eps^k|=\delta^2\eps^2$, where $\delta\in (0, 1-2\alpha)$ and $\alpha\in (0, \frac{1}{2})$ are given parameters, cf.~\eqref{square-inside}. We define the index set
	\begin{align*}
		J_\eps=\{k\in \Z^2: S_\eps^k\subset Q\},
	\end{align*}
	and observe that, for sufficiently small $\eps$, the cardinality of $J_\eps$ scales like $\eps^{-2}$, precisely, $\#J_\eps\geq c\eps^{-2}$ with a geometric constant $c>0$ depending only on $Q$. 
	Consequently,
	\begin{align}\label{99}
		|\{x \in Q: |V_\eps(x) - V(x)| > \gamma\}| & \geq |\eYrig\cap Q| =\sum_{k\in \Z^2} |(\omega_k^\eps\times \R)\cap Q| \\ &\geq  |Q|^{1/3}\sum_{k\in J_\eps} |S_\eps^k| =|Q|^{1/3} \# J_\eps \eps^2\delta^2 \geq c\delta^2|Q|^{1/3} >0\nonumber
	\end{align}		
	for all $\eps$ small enough.
	
	On the other hand, the strong convergence of $(V_\eps)_\eps$ in $L^p(U;\R^{3\times 3})$ implies its convergence in measure on $Q$ and hence, in particular,
	\begin{align*}
		|\{x \in Q : |V_\eps(x) - V(x)| > \gamma\}| \to 0 \quad \text{ as } \eps\to 0. 
	\end{align*} 
	This, however, is in contradiction with~\eqref{99}.
\end{proof}

We are now in the position to present the proof of Theorem \ref{theo:rigidity_reg}.

\begin{proof}[Proof of Theorem \ref{theo:rigidity_reg}] 
Clearly, the assumptions of Proposition \ref{prop:rigidity_general} are satisfied. We therefore know that the limit function $u$ can be represented as 
	\begin{align}\label{sigma_d_repr2}
		u(x) =  x_3\Sigma(x')  + d(x'), \quad x\in \Omega,
	\end{align} 
 with $\Sigma\in W^{1,p}(\omega;\Scal^2)$ and $d\in W^{1,p}(\omega;\R^3)$.
In order to prove that the additional condition~\eqref{second_derivatives} forces $u$ to be a rigid body motion, we show first that $\Sigma$ is constant and then conclude with the help of Lemma \ref{lem:vector_field} applied to a suitably constructed matrix-valued field. \medskip
	
	\textit{Step~1: $\Sigma$ is constant.}
To see this, we refine the proof of Proposition~\ref{prop:rigidity_general} by exchanging the estimates of Lemma~\ref{lem:neighboring_rotations} with the stronger ones from Lemma~\ref{lem:neighboring_rotations_reg}. This improves the key estimate~\eqref{estimate_frechet} by a factor $\eps^p$. 

In more detail, let us adopt the definitions and quantities introduced in the proof of Proposition~\ref{prop:rigidity_general}, up to one exception:  the parallelograms $\hori{\eps}{k}$ and $\verti{\eps}{k}$ are determined by the two parallel boundary lines
	\begin{align*}
		a_\eps^k + \eps\left(\{\tfrac{\delta}{4}\} \times (-\tfrac{\delta}{4},\tfrac{\delta}{4})\right) \quad \text{and} \quad
			a_\eps^{k+e_1} + \eps\left(\{-\tfrac{\delta}{4}\}\times (-\tfrac{\delta}{4},\tfrac{\delta}{4})\right),
	\end{align*}
	and
	\begin{align*}
		a_\eps^k + \eps\left((-\tfrac{\delta}{4},\tfrac{\delta}{4})\times\{\tfrac{\delta}{4}\}\right) \quad \text{and} \quad
			a_\eps^{k+e_2} + \eps\left((-\tfrac{\delta}{4},\tfrac{\delta}{4})\times \{-\tfrac{\delta}{4}\}\right),
	\end{align*}
	respectively, 
	see Figure~\ref{fig:parallelograms2} for an illustration in the special case when the centers $a_\eps^k$ are periodically arranged.
	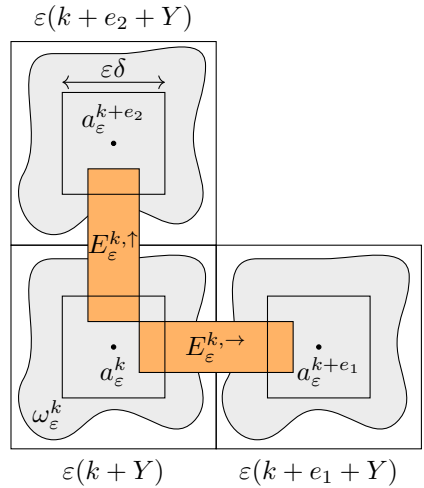
\begin{figure}[h!]
		\centering
		\begin{tikzpicture}
			\begin{scope}[scale=0.8, shift={(-0.8,-0.85)}]
	 			\draw [fill=gray!15!white] (1.5,1) [out=0,in=180] to (2.5,1.5) [out=0,in=225] to (3.8,1.3) [out=45,in=-90] to (3.8,2.6) [out=90,in=-90] to (4,3.7) [out=90,in=0] to (2.6,4) [out=180,in=0] to (1.5,4) [out=180,in=60] to (1,2) [out=240,in=180] to (1.5,1);	
	 		\end{scope}
	 		\begin{scope}[scale=0.8, shift={(-0.8,2.55)}]
	 			\draw [fill=gray!15!white] (1.5,1) [out=0,in=180] to (2.5,1.5) [out=0,in=225] to (3.8,1.3) [out=45,in=-90] to (3.8,2.6) [out=90,in=-90] to (4,3.8) [out=90,in=0] to (2.6,4) [out=180,in=0] to (1.5,4) [out=180,in=60] to (1,2) [out=240,in=180] to (1.5,1);	
	 		\end{scope}
	 		\begin{scope}[scale=0.8, shift={(2.55,-0.85)}]
	 			\draw [fill=gray!15!white] (1.5,1) [out=0,in=180] to (2.5,1.5) [out=0,in=225] to (3.8,1.3) [out=45,in=-90] to (3.8,2.6) [out=90,in=-90] to (4,3.7) [out=90,in=0] to (2.6,4) [out=180,in=0] to (1.5,4) [out=180,in=60] to (1,2) [out=240,in=180] to (1.5,1);	
	 		\end{scope}
			\begin{scope}[scale=2.7]
				\draw (0,0) rectangle (1,1);
				\draw (1,0) rectangle (2,1);
				\draw (0,1) rectangle (1,2);
				\draw (0.5,0) node [anchor=north] {\small $\eps(k+Y)$};
				\draw (1.5,0) node [anchor=north] {\small $\eps(k+e_1+Y)$};
				\draw (0.5,2) node [anchor=south] {\small $\eps(k+e_2+Y)$};
				
				\draw [fill=orange!60!white] (0.625,0.375) rectangle (1.375,0.625);
				\draw [fill=orange!60!white] (0.375,0.625) rectangle (0.625,1.375);
				\draw (0.5,1) node {\small $\verti{\eps}{k}$};
				\draw (1,0.5) node {\small $\hori{\eps}{k}$};
				
				\draw (0.25,0.25) rectangle (0.75,0.75);
				\draw (1.25,0.25) rectangle (1.75,0.75);
				\draw (0.25,1.25) rectangle (0.75,1.75);
				\draw [fill=black](0.5,0.5) circle (0.3pt);
				\draw [fill=black](1.5,0.5) circle (0.3pt);
				\draw [fill=black](0.5,1.5) circle (0.3pt);
				\draw (0.5,0.5) node [anchor=north] {\small $a_\eps^k$};
				\draw (1.55,0.5) node [anchor=north] {\small $a_\eps^{k+e_1}$};
				\draw (0.5,1.5) node [anchor=south] {\small $a_\eps^{k+e_2}$};
				\draw (0.18,0.17) node {\small $\omega_\eps^k$};
				
				\draw [<->] (0.25,1.8) --++ (0.5,0);
				\draw (0.5,1.87) node {\small $\eps\delta$};
	 		\end{scope}
		\end{tikzpicture}
		\caption{Illustration of the rectangles $\hori{\eps}{k}$ and $\verti{\eps}{k}$, connecting the horizontally and vertically neighboring squares $S_\eps^k=a_\eps^k + \eps(-\textstyle\frac{\delta}{2},\frac{\delta}{2})^2\subset \omega_\eps^k\subset \eps(k+Y)$ with a small area of overlap.}\label{fig:parallelograms2}
	\end{figure}
	Then, in analogy to Step~2 of Proposition~\ref{prop:rigidity_general}, with~\eqref{est61} in place of~\eqref{neighbor_affine}, one obtains for $\Sigma_\eps$ as in~\eqref{Sigma_e} that
	\begin{align}\label{higher_reg}
		\int_{U'} |\Sigma_{\eps}&(x' +\xi ) - \Sigma_{\eps}(x')|^p \dd x' \nonumber\\
		&\leq  C \eps^p  \left(|\xi|^p+ \eps^{p}\right) \big(\norm{\partial_1^2 u_\eps}_{L^p(\Omega;\R^3)}^p + \norm{\partial_2^2 u_\eps}_{L^p(\Omega;\R^3)}^p\big)\leq C\eps^p \left(|\xi|^p + \eps^{p}\right)
	\end{align} for any vector $\xi\in\R^2$ with $|\xi|<\frac{1}{2}\dist(U',\partial \omega)$. 
	In light of the convergence $\Sigma_\eps\to \Sigma$ in $L^p(U';\R^3)$ according to~\eqref{Sigma_convergence}, it follows from \eqref{higher_reg} that
	\begin{align*}
		\int_{U'} |\Sigma(x' +\xi ) - \Sigma(x')|^p \dd x' = 0. 
	\end{align*}
	Hence, $\Sigma$ is constant in $U'$, and  by exhaustion also in $\omega$. This implies also that 
 	$\nabla u$ is independent of $x_3$. We will prove in the next step that $\nabla u$ is constant and takes a value in $\SO(3)$. \medskip

	\textit{Step 2: Applying Lemma~\ref{lem:vector_field}.}
	Consider for $\eps>0$ the auxiliary matrix field
	\begin{align*}
		V_\eps: U'\to \R^{3\times 3}, \quad x' \mapsto \dashint_0^L \big(\nabla'u_\eps(x',x_3) | \Sigma_\eps(x')\big) \dd x_3;
	\end{align*}
	by trivial extension in $x_3$, one can also view $V_\eps$ as defined on $U$. 
	Since $\nabla u_\eps \in \SO(3)$ and $\partial_3 u_\eps = \Sigma_\eps$ a.e.~in $\eYrig\cap U$, we infer that
	\begin{align*}
		V_\eps \in \SO(3)\text{ a.e.~ in $\eYrig\cap U$}.
	\end{align*}
  	Moreover, we will see below that
	\begin{align}\label{convergence_Veps}
		V_\eps \to \nabla u \text{ in $L^p(U;\R^{3\times 3})$.}
	\end{align}
	These two observations allow us to conclude from Lemma~\ref{lem:vector_field} that
	$\nabla u\in\SO(3)$ a.e.~in $U$. Due to its gradient structure, however, the matrix field $\nabla u$ already has to coincide with a constant rotation on $U$ by Reshetnyak's theorem. This property extends to $\Omega$ by exhaustion, showing that $u$ is necessarily a rigid body motion on $\Omega$. This gives the desired statement.

	It only remains to prove~\eqref{convergence_Veps}. To this end, we start by observing that the first two columns $V_\eps':= (V_\eps e_1|V_\eps e_2)$ of $V_\eps$ satisfy
	\begin{align*}
		\int_{U'} |V'_\eps(x')|^p \dd x'= \int_{U'} \Big|\dashint_0^L\nabla' u_\eps(x',x_3) \dd x_3\Big|^p \dd x 
			\leq \frac{1}{L}\int_U |\nabla' u_\eps(x)|^p \dd x  =\frac{1}{L} \norm{\nabla'u_\eps}^p_{L^p(U;\R^{3\times 2})}.
	\end{align*} 
	due to Jensen's inequality.
	Similarly, if we recall that $u_\eps$ is twice weakly differentiable in the cross-section variables by assumption, 
	\begin{align*}
		\int_{U'} |\nabla' V'_\eps(x')|^p \dd x'\leq 4 \max_{i,j \in\{1,2\}}\int_{U'} \Big|\dashint_0^L \partial_i\partial_j u_\eps(x',x_3) \dd x_3\Big|^p \dd x 
			\leq \frac{4}{L} \max_{i,j \in\{1,2\}}\norm{\partial_i\partial_j u_\eps}^p_{L^p(U;\R^3)}.
	\end{align*} 
	In view of~\eqref{second_derivatives}, the sequence $(V_\eps')_\eps$ is therefore bounded in $W^{1,p}(U';\R^{3\times 2})$ and we can extract a non-relabeled subsequence that converges to some $V'\in W^{1,p}(U';\R^{3\times 2})$ as $\eps \to 0$, both weakly in $W^{1,p}(U';\R^{3\times 2})$ and, via Sobolev embedding, strongly to $L^p(U';\R^{3\times 2})$. 
	As $\nabla' u_\eps \weakly \nabla' u$ in $L^p(\Omega;\R^{3\times 2})$, it follows that
	\begin{align*}
		V'_\eps = \dashint_0^L\nabla'u_\eps(\cdot, x_3) \dd x_3 \weakly \dashint_0^L \nabla' u(\cdot,x_3) \dd x_3  \quad \text{ in } L^p(U';\R^3),
	\end{align*}
	and thus,
	\begin{align*}
		V' = \displaystyle\dashint_0^L \nabla' u(\cdot, x_3) \dd x_3 = \nabla'u;
	\end{align*}
	the last identity uses  that $\nabla u$ is independent of $x_3$, cf.~\eqref{sigma_d_repr2} and Step~1. 
	Together with~\eqref{Sigma_convergence} and \eqref{x_3_derivative}, we finally conclude
	\begin{align*}
		V_\eps = (V_\eps', \Sigma_\eps) \to (\nabla'u|\Sigma) = \nabla u  \text{ in } L^p(U';\R^{3\times 3}),
	\end{align*}
	which is~\eqref{convergence_Veps}. 
\end{proof}

\begin{remark}[Local bounds on the second gradients]
	Notice that assumption~\eqref{second_derivatives} in the statement of Theorem~\ref{theo:rigidity_reg} can be replaced by the weaker condition that 
	\begin{align*}
		\sup_\eps \max_{i, j\in \{1,2\}}\norm{\partial_i\partial_j u_\eps}_{L^p(K;\R^3)} <\infty
	\end{align*}
	for any compact subset $K\subset \Omega$;
	 this is immediate to verify, considering that our proof involves essentially only local arguments in the cross section.
\end{remark}

\subsection*{Acknowledgements}
The authors would like to thank Fabian Ziltener for helpful discussions related to the lifting property, and David Chiron for sharing valuable insights into the proof of
a lifting result in fiber bundles for Sobolev functions.

CK and AR were supported  by the Dutch Research Council NWO through the project TOP2.17.01. Most of this work was done while CK and AR were affiliated with Utrecht University, and CK acknowledges the partial support by the Westerdijk Fellowship program. 


\bibliographystyle{abbrv}
\bibliography{DECKAR}
\end{document}